\documentclass[11pt,onecolumn,oneside]{article}
\usepackage[a4paper, hmargin={2.8cm, 2.8cm}, vmargin={2.5cm, 2.5cm}]{geometry}
\usepackage{pdfcomment}

\usepackage[utf8]{inputenc}      		 % Input encoding
\usepackage[T1]{fontenc}                 % Output encoding (allowing fx æøåÆØÅ)
\usepackage{graphicx, color}             % Grafikpakke
\usepackage{textcomp}                    % provides extra symbols
\usepackage{amssymb}
\usepackage{amsmath}
\usepackage{amsthm}						 % Theorem environments
\usepackage{mathrsfs}
\usepackage{caption, subcaption}
\usepackage{wrapfig}
\usepackage{gensymb}
\usepackage{mathtools}
\usepackage{bm} %bold math
\usepackage{subfiles}
\usepackage[normalem]{ulem}              % Strike out text: \sout{}

\DeclareFontFamily{U}{mathx}{\hyphenchar\font45}
\DeclareFontShape{U}{mathx}{m}{n}{
      <5> <6> <7> <8> <9> <10>
      <10.95> <12> <14.4> <17.28> <20.74> <24.88>
      mathx10
      }{}
\DeclareSymbolFont{mathx}{U}{mathx}{m}{n}
\DeclareFontSubstitution{U}{mathx}{m}{n}
\DeclareMathAccent{\widecheck}{0}{mathx}{"71}
\DeclareMathAccent{\wideparen}{0}{mathx}{"75}

\usepackage{tikz}
\usetikzlibrary{cd}
\usepackage{verbatim}

\usepackage{enumerate}	      	% \begin{enumerate}[*symbol like "a)" or "i"*]
\usepackage{hyperref}			% Package for hyperreferences
\usepackage{xcolor}				% Allows for coloring of text
\hypersetup{					% Setup of the Hyperref package 
    colorlinks,
    linkcolor={red!50!black},
    citecolor={green!30!black},
    urlcolor={blue!80!black},
}
\usepackage{xkvltxp} % Necessary package for fixme package 
\usepackage[draft,nomargin,inline,index]{fixme} % Package for making labels in the thesis that needs fixing
\fxsetup{theme=color}
\newcommand{\numberthis}{\addtocounter{equation}{1}\tag{\theequation}}
\DeclareMathOperator{\supp}{supp}  % Support
\newcommand{\integral}[3]{\int_{#2} #1\,\mathrm{d}#3}

\DeclareMathOperator{\Iso}{Iso} % Isometries
\DeclareMathOperator{\Rep}{Rep}

\DeclareMathOperator{\id}{id}

\newcommand{\set}[2]{\left\{\,#1\;\middle|\; #2\,\right\}}

\makeatletter
\def\moverlay{\mathpalette\mov@rlay}
\def\mov@rlay#1#2{\leavevmode\vtop{%
   \baselineskip\z@skip \lineskiplimit-\maxdimen
   \ialign{\hfil$\m@th#1##$\hfil\cr#2\crcr}}}
\newcommand{\charfusion}[3][\mathord]{
    #1{\ifx#1\mathop\vphantom{#2}\fi
        \mathpalette\mov@rlay{#2\cr#3}
      }
    \ifx#1\mathop\expandafter\displaylimits\fi}
\makeatother

\newcommand{\net}[3]{\left(#1_{#2}\right)_{#2\in#3}}

\newcommand{\norm}[2]{\ensuremath{\left|\!\left|#1\right|\!\right|_{#2}}}
\newcommand{\normb}[2]{\ensuremath{\bigl|\!\bigl|#1\bigr|\!\bigr|_{#2}}}

\newcommand{\normdot}[1]{\ensuremath{\left|\!\left|\,\cdot\,\right|\!\right|_{#1}}}

\newcommand{\abs}[1]{\ensuremath{\left|#1\right|}}

\numberwithin{equation}{section}

\usepackage{etoolbox}
\newcommand{\addQEDstyle}[2]{\AtBeginEnvironment{#1}{\pushQED{\qed}\renewcommand{\qedsymbol}{#2}}\AtEndEnvironment{#1}{\popQED}}

\theoremstyle{plain}
\newtheorem{thm}{Theorem}
\numberwithin{thm}{section}
\newtheorem{lem}[thm]{Lemma}
\newtheorem{cor}[thm]{Corollary}
\newtheorem{prop}[thm]{Proposition}
\newtheorem{subthm}{Theorem}

\theoremstyle{definition}
\newtheorem{defn}[thm]{Definition}

\addQEDstyle{ex}{$\circ$}

\newtheorem{rk}[thm]{Remark}

\newtheoremstyle{warning}% 〈name〉
{\topsep}% 〈Space above〉
{\topsep}% 〈Space below 〉
{}% 〈Body font〉
{}% 〈Indent amount〉Indent amount: empty = no indent, \parindent = normal paragraph indent
{\bfseries}% 〈Theorem head font〉
{.}% 〈Punctuation after theorem head 〉
{.5em}% 〈Space after theorem head 〉Space after theorem head: { } = normal interword space; \newline = line break
{}% 〈Theorem head spec (can be left empty, meaning ‘normal’ )〉
\theoremstyle{warning}
\newtheorem{warning}[thm]{Warning}

\newtheoremstyle{question}% 〈name〉
{\topsep}% 〈Space above〉
{\topsep}% 〈Space below 〉
{}% 〈Body font〉
{}% 〈Indent amount〉Indent amount: empty = no indent, \parindent = normal paragraph indent
{\bfseries}% 〈Theorem head font〉
{.}% 〈Punctuation after theorem head 〉
{.5em}% 〈Space after theorem head 〉Space after theorem head: { } = normal interword space; \newline = line break
{}% 〈Theorem head spec (can be left empty, meaning ‘normal’ )〉
\theoremstyle{question}

\theoremstyle{plain}
\newtheorem{introtheorem}{Theorem}
\newtheorem{introcorollary}[introtheorem]{Corollary}

\theoremstyle{statement}
\newtheorem{introdefinition}[introtheorem]{Definition}

\title{Property $(\mathrm{T})$ for Banach algebras}

\author{Emilie Mai Elki\ae{}r and Sanaz Pooya}
\date{\today}
\begin{document}
\maketitle
\begin{abstract}
We define and study the notion of property $(\rm T)$ for Banach algebras, generalizing the one from $C^*$-algebras. For a second countable locally compact  group $G$ and a given family of Banach spaces $\mathcal E$, we prove that our Banach algebraic property $(\rm{T}_{\mathcal E})$ of the symmetrized pseudofunction algebras $F^*_{\mathcal E}(G)$ characterizes the Banach property $(\rm{T}_{\mathcal E})$ of Bader, Furman, Gelander and Monod for groups. 
In case $G$ is a discrete group and $\mathcal E$ is the class of $L^p$-spaces for $1\leq p < \infty$, we also achieve the analogue characterization using the symmetrized $p$-pseudofunction algebras $F^*_{\lambda_ p}(G)$. 
\end{abstract}

\section{Introduction}
In 1967 David Kazhdan introduced in \cite{Kazhdan1967} the notion of property $(\rm T)$ for groups in order to prove finite generation of lattices in higher rank Lie groups. Today Kazhdan's property $(\rm T)$ is a central notion of analytic group theory being used in numerous proofs including Margulis superrigidity theorem.
A second countable locally compact group $G$ has property $(\rm T)$ if, whenever a unitary representation of $G$ contains a net of almost invariant vectors, it has a non-zero invariant vector.
Lattices in higher rank semisimple Lie groups, as well as lattices in $Sp(1, n)$ enjoy this property. We refer the reader to \cite{BekkaDeLaHarpeValette} for a comprehensive treatment of Kazhdan's property $(\rm T)$.

Classically, property $(\rm T)$ concerns unitary representations and the class of complex Hilbert spaces. In \cite{BaderFurmanGelanderMonod}, Bader, Furman, Gelander and Monod extended this notion to isometric representations on Banach spaces. They showed that Kazhdan's property $(\rm T)$ is equivalent to their Banach algebraic property $(\mathrm{T}_{L^p})$ for $1 \leq p < \infty$.

On the operator algebraic side, the notion of property $(\rm T)$ was brought to von Neumann algebras by Connes in \cite{Connes1980Factorfundamentalgroup} for type  {II}$_1$-factors and later by Connes and Jones in \cite{ConnesJones1985PropertyTvonN} for general von Neumann algebras, where unitary representations were replaced by bimodules (or Connes correspondences).
Turning to $C^*$-algebras, Bekka adopted Connes's definition and formulated the notion of property $(\rm T)$ for unital $C^*$-algebras in \cite{Bekka2005PropertyTCstar} and Ng defined two versions of property $(\rm T)$ for general $C^*$-algebras in \cite{Ng2013PropertyTCstar}. Although property $(\rm T)$ for $C^*$-algebras is a younger topic compared with von Neumann algebras, it already proved useful in applications.  For example, connections with nuclearity (i.e., amenability) of $C^*$-algebras, Haagerup property for $C^*$-algebras, and property $(\rm T)$ of quantum groups were studied respectively in \cite{Brown2006PropertyTandC-algebras}, in \cite{Suzuki2013HaagerupT} and  \cite{Meng2017HaagerupCalgebrasCros}, and in \cite{KyedSoltan2012PropertyTExoticQuantumgroup}.

In this article we generalize the notion of property $(\rm T)$ to Banach algebras that possess a bounded approximate unit. Our aim is to characterize our Banach algebraic property $(\mathrm{T}_{\mathcal{E}})$ in terms of the group theoretic property $(\mathrm{T}_{\mathcal{E}})$ of Bader, Furman, Gelander and Monod \cite{BaderFurmanGelanderMonod}. Our definition, which we state below, adopts the stronger version of property $(\rm T)$ of Ng in Definition 2.1 in \cite{Ng2013PropertyTCstar} as it is stated in section 2 of \cite{BekkaNg2018PropertyTCstar}. Note that our terminology is different from that of Bekka and Ng; see Warning \ref{warning- terminalogy}.

\begin{introdefinition}(Definition \ref{def:Strong-(T_E)})
Let $\mathcal{A}$ be a Banach algebra with a bounded approximate unit and let $\mathcal E$ be a class of Banach spaces. We say that $\mathcal{A}$ has property $(\mathrm{T}_{\mathcal{E}})$ if, whenever $E\in\mathcal{E}$ is an essential $\mathcal{A}$-bimodule admitting a strictly almost $\mathcal{A}$-central net of unit vectors $\net{\xi}{i}{I}$ then there exists a net of central vectors $\net{\eta}{i}{I}$ in $E$ such that $\norm{\xi_i-\eta_i}{E}\rightarrow 0$.
\end{introdefinition}

We further define a weaker version of this property (see Definition \ref{def:weak_(T_E)}), where we only require the existence of a non-zero central vector. The weak version is referred to as \emph{weak property $(\mathrm{T}_{\mathcal{E}})$} and it is an extension of the weaker version of property $(\rm T)$ of Bekka and Ng for $C^*$-algebras (see Definition 6 in \cite{Bekka2005PropertyTCstar} for the unital case and Definition 2.1 in \cite{Ng2013PropertyTCstar} for the general case). When $\mathcal{A}$ is unital, our weak property $(\mathrm{T}_{\mathcal{E}})$ coincides with the definition suggested by Bekka in Remark 18 of \cite{Bekka2005PropertyTCstar} of a Banach space version of property $(\rm T)$ for arbitrary normed algebras.
When $\mathcal{A}$ is a $C^*$-algebra and we consider the class of complex Hilbert spaces, our definitions coincide with Ng's \cite{Ng2013PropertyTCstar} in the general case, and with Bekka's  \cite{Bekka2005PropertyTCstar} in the unital case.

Given a locally compact group $G$, we consider Banach algebras constructed from actions on a family of Banach spaces $\mathcal{E}$ referred to as 
pseudofunction algebras $F_{\mathcal{E}}(G)$. We shall consider both the pseudofunction algebra $F_{\mathcal{E}}(G)$ and its symmetrized version $F^*_{\mathcal{E}}(G)$. See paragraph \ref{prili:SymPseudo} in the preliminaries for a definition. We are particularly interested in the symmetrized pseudofunction algebra associated with the class of representations on $L^p$-spaces, $F^*_{L^p}(G)$, as well as the symmetrized version of the $p$-pseudofunction algebra, $F^*_{\lambda_p}(G)$.
Pseudofunction algebras were first studied by Herz already in the 1970's \cite{Herz1971p-spacesApplicationConvolution}, and they have subsequently been studied intensely in the context of harmonic analysis. Later they appeared in the work of Phillips, who provided an operator algebraic ground for studying these Banach algebras (see, e.g., \cite{Phillips2013SimplicityCuntz}, \cite{Phillips2017ClassificationAF}, and \cite{BlecherPhillips2019ApproximateIdentI}.) The symmetrized version of the $p$-pseudofunction algebra was first considered by Kasparov and Yu in \cite{KasparovYu_inprep}, and were studied in \cite{LiaoYu2017K} in the context of the Baum-Connes conjecture, in \cite{SameiWiersma2020Quasi-Hermitian} in the context of quasi-Hermitian groups as well as in \cite{SameiWiersma2018Exotic} in the context of exotic group $C^*$-algebras.

The main results of this article characterizes our Banach algebraic property $\mathrm{(T_{\mathcal E})}$ for symmetrized pseudofunction algebras in terms of Bader \textit{et al}'s property $\mathrm{(T_{\mathcal E})}$, for a general class of Banach spaces $\mathcal{E}$ and, in particular, for the class $L^p$ consisting of $L^p$-spaces on $\sigma$-finite measure spaces. We state the results here for the symmetrized pseudofunction algebras, but they hold for the non-symmetrized versions, as well.

% list of main results:

\begin{introtheorem} \label{thm A}
(Theorem \ref{thm:(T_E)_for_G_iff_for_FstarE(G)} and \ref{thm:(T_E)_for_G_iff_for_FstarE(G)weak})
Let $G$ be a locally compact group and let $\mathcal{E}$ be a class of Banach spaces. The following are equivalent:
\begin{enumerate}[(i)]
\item $G$ has (weak) property $(\mathrm{T}_{\mathcal{E}})$,
\item $F^*_{\mathcal{E}}(G)$ has (weak) property $(\mathrm{T}_{\mathcal{E}})$.
\end{enumerate}
\end{introtheorem}

Under the assumption that $\mathcal{E}$ is the class $L^p$-spaces, Theorem \ref{thm A} implies:

\begin{introcorollary} \label{cor B}
(Corollary \ref{cor:FLp_has_TLq_lcgp})
Let $G$ be a second countable locally compact group with property $(\mathrm{T})$ and let $1\leq p\leq2$. Then $F^*_{L^p}(G)$ has property $(\mathrm{T}_{L^q})$, for all $1\leq q\leq p$ and all $p'\leq q<\infty$, where $p'$ is the Hölder conjugate of $p$.
\end{introcorollary}

Assuming further that the group is discrete, Theorem \ref{thm A} implies: 

\begin{introcorollary} \label{cor C}
(Corollary \ref{cor:FLp_has_TLq_disgp})
Let $\Gamma$ be a discrete group with property $(\rm T)$ and let $1\leq p\leq2$. Then $F^*_{L^p}(\Gamma)$ has property $(\mathrm{T}_{L^q})$, for all $1\leq q<\infty$.
\end{introcorollary}

Continuing with the assumptions that $\Gamma$ is a discrete group and $\mathcal{E}$ is the class  $L^p$-spaces,
we obtain, in addition to the characterization from Theorem \ref{thm A}, a characterization of property $(\mathrm{T}_{L^p})$ in terms of the symmetrized $p$-pseudofunction algebra $F^*_{\lambda_p}(\Gamma)$.

\begin{introtheorem}\label{thm C}
(Theorem \ref{TLp_for_IN_gps})
Let $\Gamma$ be a discrete group and let $1\leq p<\infty$. The following are equivalent:
\begin{enumerate}[(i)]
\item $\Gamma$ has property $(\mathrm{T}_{L^p})$,
\item $F^*_{L^p}(\Gamma)$ has property $(\mathrm{T}_{L^p})$,
\item $F^*_{\lambda_p}(\Gamma)$ has property $(\mathrm{T}_{L^p})$.
\end{enumerate}
\end{introtheorem}
Theorem \ref{thm C} holds for the non-symmetrized versions of the algebras, as well. The proof relies on an $L^p$-analogue of Fell's absorption principle; see Lemma \ref{prop:LpFellsAbsorption}.

This paper is organized as follows. Chapter 2 contains preliminaries on actions of Banach algebras, multiplier algebras, symmetrized pseudofunction algebras as well as property $(\mathrm{T}_{\mathcal{E}})$ for actions of groups on Banach spaces. 
In chapter 3 we prove Theorem \ref{thm A} and its implications Corollary \ref{cor B} and Corollary \ref{cor C}.
In chapter 4 we prove Theorem \ref{thm C}. Further, we show that weak property $(\mathrm{T}_{SL^p})$ is stronger than property $(\mathrm{T}_{L^p})$ for  discrete groups.

\section{Preliminaries}
\subsection{Actions of Banach algebras on Banach spaces.}
For a Banach algebra $\mathcal{A}$, we denote by $\mathcal{A}^{\mathrm{op}}$ its opposite algebra, i.e., the Banach algebra with the same underlying Banach space as $\mathcal{A}$ but with multiplication in reversed order.

\begin{defn}
Let $\mathcal{A}$ and $\mathcal{B}$ be Banach algebras and $E$ a Banach space. A \emph{left action} of $\mathcal{A}$ on $E$ is a contractive representation of $\mathcal{A}$ on $E$. A \emph{right action} of $\mathcal{A}$ on $E$ is a contractive representation of $\mathcal{A}^{\mathrm{op}}$ on $E$. We say that $E$ is a \emph{left (right) $\mathcal{A}$-module} if it carries a left (right) action of $\mathcal{A}$. Further, $E$ is an $\mathcal{A}$-$\mathcal{B}$-bimodule if it carries a left action $\varphi$ of $\mathcal{A}$ and a right action $\psi$ of $\mathcal{B}$ with commuting ranges. We write
\[
a\cdot\xi\cdot b=\psi(b)\varphi(a)\xi ,\qquad\mbox{for } a\in\mathcal{A},b\in\mathcal{B}\mbox{ and } \xi \in E.\qedhere
\]
An $\mathcal{A}$-$\mathcal{A}$-bimodule is simply referred to as an $\mathcal{A}$-bimodule.
\end{defn}

\begin{defn}
An $\mathcal{A}$-$\mathcal{B}$-bimodule $E$ is said to be \emph{essential} if the span of $\mathcal{A}\cdot E\cdot\mathcal{B}$ is dense in $E$. 
Further, $E$ is said to be \emph{faithful} if, whenever $\xi\in E$ satisfies $a\cdot\xi\cdot b=0$, for all $a\in\mathcal{A}$ and $b\in\mathcal{B}$, then $\xi=0$.
\end{defn}

\begin{rk}
Let $\mathcal{A}$ be a Banach algebra with a bounded approximate unit $(u_i)_i$. Then any essential $\mathcal{A}$-bimodule will necessarily be faithful. Indeed, if $E$ is an essential $\mathcal{A}$-bimodule it follows from Cohen's factorization theorem \cite{Cohen1959Factorization} that $E$ is pseudo-unital, i.e., $E=\mathcal{A}\cdot E\cdot\mathcal{A}$. Suppose $\xi\in E$ and $a,b\in\mathcal{A}$ are such that $ca\cdot\xi\cdot bd=0$, for all $c,d\in\mathcal{A}$, then $a\cdot\xi\cdot b=0$. Then, for each $i$,
\begin{align*}
\norm{a\cdot\xi\cdot b}{}&=\norm{a\cdot\xi\cdot b-u_ia\cdot\xi\cdot b}{}+\norm{u_ia\cdot\xi\cdot b-u_ia\cdot\xi\cdot bu_i}{}\\
&\leq \norm{a-u_ia}{}\norm{\xi\cdot b}{}+\sup_j\norm{u_j}{}\norm{a\cdot \xi}{}\norm{b-bu_i}{}.
\end{align*}
As the right-hand side tends to zero in $i$, we see that $a\cdot\xi\cdot b=0$. Since all elements of $E$ admits a factorization in the form $a\cdot\xi\cdot b$, it follows that $E$ is faithful.

%if $\mathcal{A}$ admits a bounded approximate unit and $E$ is an essential $\mathcal{A}$-bimodule then it follows from Cohen's factorization theorem \cite{Cohen1959Factorization} that $E$ is pseudo-unital, i.e., $E=\mathcal{A}\cdot E\cdot\mathcal{A}$. 
\end{rk}

The Banach $^*$-algebra $L^1(G)$, for a locally compact group $G$, plays an important role among the Banach algebras we consider. Since $L^1(G)$ always carries a bounded approximate unit, it is essential and faithful as a bimodule over itself. The following fact connecting its contractive representation theory with the isometric representation theory of $G$ is folklore:

\begin{prop}\label{prop:isom_rep_of_G_is_1-to-1_contr_rep_of_L1(G)}
Let $G$ be a locally compact group and $E$ a Banach space. There is a 1-1 correspondence between non-degenerate, contractive representations of $L^1(G)$ on $E$ and strongly continuous isometric representations of $G$ on $E$.
\end{prop}
The Banach $^*$-algebra $L^1(G)$ is in the following precise sense its own opposite, as may be easily verified:
\begin{prop}\label{prop:tilde_map_stariso_L1}
The map $\widetilde{\square}:L^1(G)\rightarrow L^1(G)^{\mathrm{op}}$ given by
\[
\widetilde{f}(s)=\Delta(s^{-1})f(s^{-1}),\qquad\mbox{for }s\in G, \numberthis\label{eqn:def:tilde-map}
\]
is an isometric $^*$-isomorphism.
\end{prop}

\subsection{Multipliers of Banach algebras.}
We refer the reader to Section 2 in \cite{Daws2010Multipliers} for details and more results on multipliers of modules. 
\begin{defn}
Let $\mathcal{A}$ be a Banach algebra. A \emph{multiplier} of $\mathcal{A}$ is a pair of maps $(L,R)$ from $\mathcal{A}$ to itself satisfying
\[
    aL(b)=R(a)b,\qquad\mbox{for all }a,b\in\mathcal{A}.
\]
We denote by $\mathrm{M}(\mathcal{A})$ the set of all multipliers on $\mathcal{A}$. This is a linear space with addition and scalar multiplication defined in the obvious way.
\end{defn}

There is a canonical linear map $\mathcal{A}\rightarrow \mathrm{M}(\mathcal{A})$ defined by assigning to each $a_0\in \mathcal{A}$ a pair of maps $L_{a_0}, R_{a_0}:\mathcal{A}\rightarrow \mathcal{A}$ as follows:
\[
L_{a_0}:a\mapsto a_0a\qquad\mbox{and}\qquad R_{a_0}: a\mapsto aa_0,\qquad\mbox{for }a\in\mathcal{A}.
\]
If $\mathcal{A}$ is faithful as a bimodule over itself, this map is injective. If $\mathcal{A}$ is unital, it is surjective. In general, it need neither be injective nor surjective.\\

When $\mathcal{A}$ is faithful as a bimodule over itself, one can show that $\mathrm{M}(\mathcal{A})$ embeds linearly into $\mathscr{B}(\mathcal{A})\oplus_{\infty}\mathscr{B}(\mathcal{A})$. Hence, $\mathrm{M}(\mathcal{A})$ inherits the \emph{strict topology}, i.e., the locally convex topology generated by the family of seminorms $(L,R)\mapsto\norm{L(a)}{}+\norm{R(a)}{}$, where $a\in\mathcal{A}$. The next proposition may be verified with routine arguments:

\begin{prop}\label{prop:DC_is_strictly_closed}
Let $\mathcal{A}$ be a Banach algebra and assume that $\mathcal{A}$ is faithful as a bimodule over itself. Then $\mathrm{M}(\mathcal{A})$ is a strictly closed subspace of $\mathscr{B}(\mathcal{A})\oplus_{\infty}\mathscr{B}(\mathcal{A})$.
\end{prop}

The following lemma concerning strictly compact subsets of $\mathrm{M}(\mathcal{A})$ is Lemma 8 in \cite{BekkaNg2018PropertyTCstar} put in our more general setting. The proof is the same and so we omit it.

\begin{lem}\label{lem:str_cpt_properties}
Let $\mathcal{A}$ be a Banach algebra which is faithful as a bimodule over itself and let $S$ be a non-empty strictly compact subset of $\mathrm{M}(\mathcal{A})$. Then $S$ satisfies the following two properties:
\begin{enumerate}[(i)]
    \item $S$ is norm-bounded,
    \item for any element $a_0\in\mathcal{A}$ and any $\varepsilon>0$, there exist a finite number of elements $x_1,\ldots,x_n\in S$ such that, for every $x\in S$, there is a $k\in\{1,\ldots,n\}$ for which $\norm{x\cdot a_0-x_k\cdot a_0}{\mathcal{A}}<\varepsilon$.
\end{enumerate} 
\end{lem}

Let $\mathcal{A}$ be a Banach algebra with a bounded approximate unit so that any essential $\mathcal{A}$-bimodule is automatically faithful. Given another Banach algebra $\mathcal{B}$ and a bounded homomorphism $\varphi:\mathcal{A}\rightarrow\mathcal{B}$, the Banach algebra $\mathcal{B}$ becomes in a natural way a bimodule over $\mathcal{A}$. Moreover, if $\varphi$ has dense range, the image of any bounded approximate unit on $\mathcal{A}$ is a bounded approximate unit on $\mathcal{B}$. The following is a special case of Theorem 2.8 in \cite{Daws2010Multipliers}:

\begin{thm}\label{thm:extending_hom_to_DC} 
Let $\mathcal{A}$ be a Banach algebra with a bounded approximate unit, let $\mathcal{B}$ be another Banach algebra and let $\varphi:\mathcal{A}\rightarrow\mathcal{B}$ be a bounded homomorphism with dense range. There is a unique extension $\Phi:\mathrm{M}(\mathcal{A})\rightarrow\mathrm{M}(\mathcal{B})$ and this extension is strictly continuous.
\end{thm}

We close this part with some remarks on the multiplier algebra of $L^1(G)$, for a locally compact group $G$.
For $s\in G$, we denote by $L_s$ and $R_s$ the left, respectively, right translation operators on $L^1(G)$. Precisely, for $s,t\in G$ and $f\in L^1(G)$,
\begin{gather*}
  L_sf(t)=f(s^{-1}t), \quad
R_sf(t)=f(ts)  
\end{gather*}
Then $(L_s,\Delta(s^{-1})R_{s^{-1}})$ is a multiplier of $L^1(G)$. This gives rise to a multiplicative embedding $G\hookrightarrow\mathrm{M}(L^1(G))$, and this embedding is continuous when $\mathrm{M}(L^1(G))$ is equipped with the strict topology. Thus, given an essential $L^1(G)$-bimodule $E$, the group $G$, respectively, its opposite $G^{\mathrm{op}}$ act on $E$ via the extension of the bimodule structure to the multiplier algebra. The next proposition, which may be verified with a straight forward computation, shows that these actions agree with the actions we already have from the $L^1(G)$-bimodule structure via Proposition \ref{prop:isom_rep_of_G_is_1-to-1_contr_rep_of_L1(G)}.

\begin{prop}\label{prop:DC_bimodule_structure_encodes_gp_action}
Let $G$ be a locally compact group and let $E$ be an essential $L^1(G)$-bimodule with left action $\varphi$ and right-action $\psi$. For each $\xi\in E$ and $s\in G$,
\[
(L_s,\Delta(s^{-1})R_{s^{-1}})\cdot\xi=\varphi(s)\xi\qquad\mbox{and}\qquad 
\xi\cdot(L_s,\Delta(s^{-1})R_{s^{-1}})=\psi(s)\xi
\]
\end{prop}

\subsection{Symmetrized pseudofunction algebras} \label{prili:SymPseudo}
Let $G$ be a locally compact group. Given a class of Banach spaces $\mathcal{E}$, denote by $\Rep_{\mathcal{E}}(G)$ the class of strongly continuous isometric representations of $G$ on a Banach space in $\mathcal{E}$. For a subclass $\mathcal{R}$ of $\Rep_{\mathcal{E}}(G)$, we define a seminorm on $L^1(G)$ by setting
\[
\norm{f}{\mathcal{R}}=\sup\set{\norm{\pi(f)}{}}{\pi\in\mathcal{R}}.
\]
Set $I_{\mathcal{R}}=\bigcap_{\pi\in\mathcal{R}}\ker(\pi)$. This is a closed 2-sided ideal in $L^1(G)$, and so, the quotient $L^1(G)/I_{\mathcal{R}}$ inherits the algebra structure from $L^1(G)$. We denote by $F_{\mathcal{R}}(G)$ the completion of $L^1(G)/I_{\mathcal{R}}$ with respect to the norm induced by $\normdot{\mathcal{R}}$. This is a Banach algebra with multiplication extending the convolution product on $L^1(G)$; we refer to it as the \emph{Banach algebra of $\mathcal{R}$-pseudofunctions}. When $\mathcal{R}$ is all of $\Rep_{\mathcal{E}}(G)$, we shall denote by $F_{\mathcal{E}}(G)$ the resulting Banach algebra. When $\mathcal{R}$ consists of only one representation, say $\pi$, we simply write $F_{\pi}(G)$. Accordingly, we refer to these Banach algebras as algebras of \emph{$\mathcal{E}$-pseudofunctions}, respectively, \emph{$\pi$-pseudofunctions}.

Well-known examples of pseudofunction algebras include the universal and the reduced group $C^*$-algebras, $C^*(G)$ and $C^*_r(G)$, respectively. In the notation introduced above, the former is the pseudofunction algebra $F_{\mathcal{H}}(G)$, where $\mathcal{H}$ is the class of complex Hilbert spaces, and the latter is $F_{\lambda}(G)$, where $\lambda$ the left-regular representation of $G$. Further, for $1\leq p<\infty$ and $\lambda_p$ the left-regular representation of $G$ on $L^p(G)$, $F_{\lambda_p}(G)$ is the Banach algebra of $p$-pseudofunctions, which goes back to work of Herz and it is often denoted by $PF_p(G)$. This Banach algebra also appeared in work of Phillips, e.g. \cite{Phillips2013SimplicityCuntz}, where it is denoted by $F^p_r(G)$ to emphasize its connection to the reduced group $C^*$-algebra.

\begin{rk}\label{rk:isom_rep_of_G_is_1-to-1_contr_rep_of_F_R(G)}
It is easy to see that if $\pi$ is an isometric representation of $G$, and $\pi$ lies in the class $\mathcal{R}$, then $\pi$ extends to a non-degenerate contractive representation of $F_{\mathcal{R}}(G)$. That is, $F_{\mathcal{R}}(G)$ is universal for $\mathcal{R}$ in the same way that $C^*(G)$ is universal for all unitary representations of $G$. Conversely, by Proposition \ref{prop:isom_rep_of_G_is_1-to-1_contr_rep_of_L1(G)}, if $\pi$ is a non-degenerate contractive representation of $F_{\mathcal{R}}(G)$ then $\pi$ is the extension of an integrated form of an isometric representation of $G$. However, we are not guaranteed that $\pi$ lies in the class $\mathcal{R}$.
\end{rk}

The involution on $L^1(G)$ need not extend to $F_{\mathcal{R}}(G)$, and so, $F_{\mathcal{R}}(G)$ is in general only a Banach algebra and not necessarily a Banach $^*$-algebra. However, if the class $\mathcal{R}$ is closed under duality, the involution on $L^1(G)$ does extend. Recall that if $\pi$ is an isometric representation of $G$ on a Banach space $E$, its dual representation $\pi^*$ is the isometric representation on the dual Banach space $E^*$ given, for $t\in G$, $\eta\in E^*$ and $x\in E$, by
\[
\left(\pi^*(t)\eta\right)(x)=\eta\left(\pi(t^{-1})x\right).
\]
We say that the class $\mathcal{R}$ is closed under duality if $\pi^*\in\mathcal{R}$ whenever $\pi\in\mathcal{R}$. Proposition \ref{prop:extending_the_involution_to_pseudofunctions} below is proven in a special case in Proposition 4.2 in \cite{SameiWiersma2020Quasi-Hermitian}. The proof in the generality stated here is essentially the same, and so, we omit it.

\begin{prop}\label{prop:extending_the_involution_to_pseudofunctions}
Let $\mathcal{R}$ be a class of continuous isometric representations of $G$ closed under duality. Then the involution on $L^1(G)$ is an isometry with respect to the norm induced by $\mathcal{R}$.
\end{prop}

For a class $\mathcal{R}$ of strongly continuous isometric representations of $G$, denote by $\mathcal{R}^*$ the smallest class of strongly continuous isometric representations of $G$ which is closed under duality and which contains $\mathcal{R}$. We denote by $F^*_{\mathcal{R}}(G)$ the completion of $L^1(G)/I_{\mathcal{R}^*}$ with respect to the norm
\[
\norm{f}{F^*_{\mathcal{R}}(G)}=\sup\set{\norm{\pi(f)}{}}{\pi\in\mathcal{R}^*}
\]
By Proposition \ref{prop:extending_the_involution_to_pseudofunctions}, $F^*_{\mathcal{R}}(G)$ is a Banach $^*$-algebra, and we shall refer to it as the \emph{symmetrized Banach $^*$-algebra of $\mathcal{R}$-pseudofunctions}. As in the non-symmetrized setting, when $\mathcal{R}$ is all of $\Rep_{\mathcal{E}}(G)$ or when $\mathcal{R}$ consists of a single representation $\pi$, we shall write $F^*_{\mathcal{E}}(G)$, respectively, $F^*_{\pi}(G)$, and we refer to these accordingly.

\begin{rk}
Let $\mathcal{E}_{\mathrm{ref}}$ be the class of all reflexive Banach spaces. For a subclass $\mathcal{E}\subset\mathcal{E}_{\mathrm{ref}}$, denote by $\mathcal{E}'$ the class consisting of the Banach spaces which are dual to the Banach spaces in $\mathcal{E}$. Let $\mathcal{R}$ be a subclass of $\Rep_{\mathcal{E}}(G)$ and denote by $\mathcal{R}'$ the subclass of $\Rep_{\mathcal{E}'}(G)$ consisting of representations which are dual to the representations in $\mathcal{R}$. Then $\mathcal{R}^*=\mathcal{R}\cup\mathcal{R}'$. In this case, a straight forward computation gives
\[
\norm{f}{F^*_{\mathcal{R}}(G)}=\max\left\{\norm{f}{F_{\mathcal{R}}(G)},\normb{\widetilde{f}}{F_{\mathcal{R}}(G)}\right\}
\]
\end{rk}

\begin{rk}
Similar to $L^1(G)$, $F^*_{\mathcal{R}}(G)$ is self-opposite via the map defined in equation (\ref{eqn:def:tilde-map}). This need not be true for general pseudofunction algebras.
\end{rk}

\subsection{Property $(\rm T)$ for groups acting on Banach spaces.}
% à la Bader, Furman, Gelander, and Monod
For a Banach space $E$, denote by $\Iso(E)$ the group of surjective linear isometries on $E$. A strongly continuous isometric representation of a locally compact group $G$ on a Banach space $E$ is a strongly continuous group homomorphism $\pi:G\rightarrow\Iso(E)$. Given such a representation $(\pi,E)$, we denote by $E^{\pi}$ the subspace of $G$-invariant vectors. In \cite{BaderFurmanGelanderMonod}, Bader, Furman, Gelander and Monod define property $(\mathrm{T}_{\mathcal{E}})$ as follows:
\begin{defn}\label{def:(T_E)_BFGM} Definition 1.1 in \cite{BaderFurmanGelanderMonod}
    Let $\mathcal{E}$ be a class of Banach spaces. A locally compact group $G$ has property $(\mathrm{T}_{\mathcal{E}})$ if, for any continuous isometric representation $(\pi,E)$ with $E$ in the class $\mathcal{E}$, the quotient representation $\pi'\colon G \rightarrow \Iso(E/{E}^{\pi})$ does not have almost $G$-invariant vectors. If $\mathcal{E}$ consists of a single Banach space $E$, we write $(\mathrm{T}_{E})$ instead of $(\mathrm{T}_{\mathcal{E}})$.
\end{defn}
\begin{rk} \label{rk: T implies TLp}
    We recall from Theorem A in \cite{BaderFurmanGelanderMonod} that, for a second countable locally compact group $G$, Kazhdan's property $(\mathrm{T})$ coincides with property $(\mathrm{T}_{L^p(\mu)})$, for any $\sigma$-finite measure $\mu$ and any $1\leq p < \infty$. 
\end{rk}

We shall use the following alternative definition of property $(\mathrm{T}_{\mathcal{E}})$, which is equivalent to the definition above by Lemma 18 in \cite{Tanaka2017PropertyT}. There, the lemma is stated for second countable locally compact groups, but the additional assumption that the group is second countable can be dropped.

\begin{defn}\label{def:(T_E)_superrefl}
Let $\mathcal{E}$ be a class of Banach spaces. A locally compact group $G$ has \emph{property $(\mathrm{T}_{\mathcal{E}})$} if, whenever $(\pi,E)$ is a strongly continuous isometric representation of $G$ with $E$ in the class $\mathcal{E}$ admitting a net of almost invariant unit vectors $\net{\xi}{i}{I}$, there exists a net of $G$-invariant vectors $\net{\eta}{i}{I}$ such that $\norm{\xi_i-\eta_i}{E}\rightarrow0$. 
\end{defn}
%\begin{rk}
%Parallel to what is the case for Kazhdan's property $(\mathrm{T})$, there is a quantitative version of Definition \ref{def:(T_E)_superrefl} (see Remark 2.11 in \cite{BaderFurmanGelanderMonod}). Because we shall not be interested in the size of Kazhdan type constants in this paper, we stick to the qualitative formulation.
%\end{rk}

It may be tempting to define the Banach space version of property $(\mathrm{T})$ for groups parallel to the often used definition of Kazhdan's property $(\mathrm{T})$ which only requires the existence of a non-zero invariant vector. 
A priori, this is a weaker property. 
%We shall refer to it as \emph{weak property $(\mathrm{T}_{\mathcal{E}})$}.

\begin{defn}\label{def:group_weak_(T_E)}
Let $\mathcal{E}$ be a class of Banach spaces. A locally compact group $G$ has \emph{weak property $(\mathrm {T}_{\mathcal{E}})$} if any strongly continuous isometric representation $(\pi,E)$ with $E\in\mathcal{E}$ admitting almost invariant vectors has a non-zero $G$-invariant vector.
\end{defn}

It is well-known that property $(\mathrm{T}_{\mathcal{H}})$ is equivalent to weak property $(\mathrm{T}_{\mathcal{H}})$ when $\mathcal{H}$ is the class of complex Hilbert spaces, in which we recover Kazhdan's property $(\rm T)$. A bit more generally, Proposition \ref{prop:on_eq_of_(T_E)_and_weak(T_E)} gives two sufficient conditions on the class $\mathcal{E}$ for the equivalence of property $(\mathrm{T}_{\mathcal{E}})$ and weak property $(\mathrm{T}_{\mathcal{E}})$. This may be of independent interest. The conditions are well-known to experts, but to our knowledge, they do not appear explicitly in the literature. 
\begin{prop}\label{prop:on_eq_of_(T_E)_and_weak(T_E)}
For any second countable locally compact group $G$ and any class of Banach spaces $\mathcal{E}$, property $(\mathrm{T}_{\mathcal{E}})$ implies weak property $(\mathrm{T}_{\mathcal{E}})$. The converse is true if $\mathcal{E}$ satisfies either one of the following properties:
\begin{enumerate}[(i)]
    \item $\mathcal{E}$ is stable under quotients,
    \item $\mathcal{E}$ is a class of superreflexive Banach spaces stable under taking complemented subspaces.
\end{enumerate}
\end{prop}
\begin{proof}
Assume $G$ does not have property $(\mathrm{T}_{\mathcal{E}})$. We can then find a continuous isometric representation $(\pi,E)$ of $G$ with $E$ in $\mathcal{E}$ such that the quotient $E/E^{\pi}$ admits a net of almost invariant vectors. However, $E/E^{\pi}$, by construction,  has no non-zero $G$-invariant vectors. If there is an isometric representation of $G$ on a space $F$ in $\mathcal{E}$ and a bounded equivariant isomorphism from $F$ to $E/E^{\pi}$, it follows that $G$ does not have property $(\mathrm{T}_{\mathcal{E}})$. This is trivially the case if $\mathcal{E}$ is stable under quotients. If $\mathcal{E}$ consists of superreflexive Banach spaces then $E^{\pi}$ is a complemented subspace and its complement is isomorphic to the quotient $E/E^{\pi}$ (see, e.g., Proposition 2.6 in \cite{BaderFurmanGelanderMonod}). If $\mathcal{E}$ furthermore is stable under taking complemented subspaces, we have a contraction from a space in $\mathcal{E}$ to the quotient $E/E^{\pi}$. Hence, if the class $\mathcal{E}$ satisfies either of the conditions (i) or (ii), we see that if $G$ does not have property $(\mathrm{T}_{\mathcal{E}})$ it does not have weak property $(\mathrm{T}_{\mathcal{E}})$ either. 
\end{proof}

\section{Property $\mathrm{(T_{\mathcal{E}})}$ for Banach algebras}
In this section we define property $\mathrm{(T_{\mathcal E})}$, as well as a weaker version of it, for a Banach algebra acting on a family of Banach spaces $\mathcal E$. Our definitions extend that of Bekka-Ng  in \cite{BekkaNg2018PropertyTCstar} for (not necessarily unital) $\mathrm C ^*$-algebras to a Banach algebraic setting. 

The main result in this section, Theorem \ref{thm:(T_E)_for_G_iff_for_FstarE(G)}, relates property $(\mathrm{T}_{\mathcal E})$ of Bader \textit{et al} for locally compact groups to that of its symmetrized Banach $^*$-algebra of $\mathcal {E}$-pseudofunctions. 
Denoting by $L^p$ the class of all $L^p$-spaces on $\sigma$-finite measure spaces, we show that property $(\mathrm{T}_{L^p})$ for $F^*_{L^ p}(G)$ is independent on the parameter $p$. 
Finally, under the assumption that $G$ has Kazhdan's property $(\rm T)$, we obtain property $(\mathrm{T}_{L^q})$ for $F^*_{L^ p}(G)$.\\

Let $\mathcal{A}$ be a Banach algebra and $E$ an $\mathcal{A}$-bimodule. We say that $\xi\in E$ is \emph{$\mathcal{A}$-central} if, for all $a\in\mathcal{A}$, $a\cdot\xi=\xi\cdot a$.
The set of all such elements constitute a closed subspace of $E$, which we denote by $E^{\mathcal{A}}$.
A net $\net{\xi}{i}{I}$ in $E$ is said to be \emph{almost $\mathcal{A}$-central} if, for every finite subset $F\subset\mathcal{A}$ and every $\varepsilon>0$, there is an index $i_0\in I$ such that, for all $i\succcurlyeq i_0$,
\[
\sup_{a\in F}\norm{a\cdot\xi_i-\xi_i\cdot a}{E}<\varepsilon.
\]
A net $\net{\xi}{i}{I}$ in $E$ is said to be \emph{strictly almost $\mathcal{A}$-central} if, for every strictly compact subset $S\subset\mathrm{M}(\mathcal{A})$ and every $\varepsilon>0$, there is an index $i_0\in I$ such that, for all $i\succcurlyeq i_0$,
\[
\sup_{x\in S}\norm{x\cdot\xi_i-\xi_i\cdot x}{E}<\varepsilon.
\]

We can now state our two main definitions:

\begin{defn}\label{def:Strong-(T_E)}
Let $\mathcal{A}$ be a Banach algebra with a bounded approximate unit and let $\mathcal{E}$ be a class of Banach spaces. We say that $\mathcal{A}$ has \emph{property $(\mathrm{T}_{\mathcal{E}})$} if, whenever $E\in\mathcal{E}$ is an essential $\mathcal{A}$-bimodule admitting a net of strictly almost $\mathcal{A}$-central unit vectors $\net{\xi}{i}{I}$, then there exists a net $\net{\eta}{i}{I}$ of $\mathcal{A}$-central vectors such that $\norm{\xi_i-\eta_i}{E}\rightarrow 0$.
\end{defn}
\begin{defn}\label{def:weak_(T_E)}
Let $\mathcal{A}$ be a Banach algebra with a bounded approximate unit and let $\mathcal{E}$ be a class of Banach spaces. We say that $\mathcal{A}$ has \emph{weak property $(\mathrm{T}_{\mathcal{E}})$} if, whenever $E\in \mathcal{E}$ is an essential $\mathcal{A}$-bimodule admitting a net of strictly almost $\mathcal{A}$-central unit vectors, then $E$ contains a non-zero $\mathcal{A}$-central vector.
\end{defn}
It is immediately clear that property $(\mathrm{T}_{\mathcal{E}})$ implies weak property $(\mathrm{T}_{\mathcal{E}})$.
\begin{rk}
We have chosen to restrict the definitions of the two versions of property $(\mathrm{T}_{\mathcal{E}})$ to Banach algebras possessing a bounded approximate unit. This includes, in particular, all the pseudofunction algebras. The definitions, however, are sensible for any Banach algebra which is faithful as a bimodule over itself.
\end{rk}

\begin{rk}
When $\mathcal{A}$ is a $C^*$-algebra and $\mathcal{H}$ is the class of complex Hilbert spaces, our property $(\mathrm{T}_{\mathcal{H}})$ for $\mathcal{A}$ recovers the stronger version of property $(\rm T)$ of Bekka and Ng while our weak property $(\mathrm{T}_{\mathcal{H}})$ recovers the weaker version of their property $(\rm T)$. Indeed, the assumption that the bimodules are essential forces the extension of the bimodule structure to the multiplier algebra to be unital, and unital contractive algebra homomorphisms between $C^*$-algebras are necessarily $^*$-preserving. Hence, we shall refer to (weak) property $(\mathrm{T}_{\mathcal{H}})$ simply as (weak) property $(\rm T)$.
\end{rk}
\begin{warning} \label{warning- terminalogy}
Our terminology differs from that of Bekka and Ng for $C^*$-algebras: While Bekka and Ng use the terms \emph{strong property $(\rm T)$} and \emph{property $(\rm T)$} for the stronger, respectively, weaker version, we prefer the terms \emph{property $(\rm T)$} and \emph{weak property $(\rm T)$}. In particular, \emph{property $(\rm T)$} is the weaker version for them and the stronger version for us. As we shall see in a moment, our terminology is better aligned with the terminology on the group level. Moreover, we avoid confusion with the established notion of strong property $(\rm T)$ for groups.
\end{warning}

\begin{rk}\label{rk:Strong-(T_E)_when_subsp_is_compl}
If $E^{\mathcal{A}}$ is a complemented subspace of $E$ and $P$ is the projection onto $E^{\mathcal{A}}$ along its complement, we may take $\eta_i=P\xi_i$ in the definition of property $(\rm{T}_{\mathcal{E}})$.
\end{rk}

In his definition of property $(\rm T)$ for a unital $C^*$-algebra $\mathcal{A}$ (see Definition 6 in \cite{Bekka2005PropertyTCstar}), Bekka considered $\mathcal{A}$-bimodules admitting a net of almost $\mathcal{A}$-central unit vectors rather than strictly almost $\mathcal{A}$-central unit vectors. When $\mathcal{A}$ is unital, his definition and that of Ng coincides (see Proposition 2.5(b) in \cite{Ng2013PropertyTCstar}). In the more general setting of property $(\mathrm{T}_{\mathcal{E}})$ for Banach algebras, the same phenomenon happens.

\begin{prop}\label{prop:TE_for_unital_algebras}
Let $\mathcal{A}$ be a Banach algebra with a bounded approximate unit and let $\mathcal{E}$ be a class of Banach spaces. Assume $\mathcal{A}$ satisfies the following property: whenever $E\in\mathcal{E}$ is an essential $\mathcal{A}$-bimodule admitting a net $\net{\xi}{i}{I}$ of almost $\mathcal{A}$-central unit vectors then $E$ contains a net $\net{\eta}{i}{I}$ of central vectors such that $\norm{\xi_i-\eta_i}{E}\rightarrow0$. Then $\mathcal{A}$ has property $(\mathrm{T}_{\mathcal{E}})$. The converse is true if $\mathcal{A}$ is unital.
\end{prop}
\begin{proof}
Consider the canonical embedding $\varsigma:\mathcal{A}\hookrightarrow\mathrm{M}(\mathcal{A})$.
If $F\subset\mathcal{A}$ is a finite subset then $\varsigma(F)$ is a finite, hence strictly compact, subset of $\mathrm{M}(\mathcal{A})$. Now, for any $\xi\in E$, we have
\[
\sup_{a\in F}\norm{a\cdot\xi-\xi\cdot a}{E}=\sup_{x\in \varsigma(F)}\norm{x\cdot\xi-\xi\cdot x}{E}.
\]
Hence, any net of strictly almost $\mathcal{A}$-central vectors is automatically also a net of almost $\mathcal{A}$-central vectors. It is thus clear that property $(\mathrm{T}_{\mathcal{E}})$ for $\mathcal{A}$ is implied by the mentioned property.

For the converse implication, assume that $\mathcal{A}$ is unital. Let $S$ be any strictly compact subset of $\mathrm{M}(\mathcal{A})$. Because $\mathcal{A}$ is unital, any element in $\mathrm{M}(\mathcal{A})$ is of the form $(L_a,R_a)$, for some $a\in\mathcal{A}$.
Thus, $S=\varsigma(F)$ for some (not necessarily finite) subset $F$ of $\mathcal{A}$. For a given $\varepsilon>0$, we apply Lemma \ref{lem:str_cpt_properties}(ii) with the identity of $\mathcal{A}$ in place of $a_0$ to obtain a finite number of elements $a_1,\ldots,a_n\in F$ such that, for any $a\in F$, there is a $k\in\{1,\ldots,n\}$ for which
\[
\norm{a-a_k}{\mathcal{A}}=\norm{L_a(1_{\mathcal{A}})-L_{a_k}(1_{\mathcal{A}})}{\mathcal{A}}<\varepsilon.\numberthis\label{eq:TE_for_unital_algebras_approx_a}
\]
Let $E$ be an $\mathcal{A}$-bimodule. For $a\in F$, take $k$ such that equation (\ref{eq:TE_for_unital_algebras_approx_a}) holds. Then,
\[
\norm{a\cdot\xi-\xi\cdot a}{E}
\leq \norm{a_k\cdot\xi-\xi\cdot a_k}{E}+2\varepsilon\norm{\xi}{E},
\]
for any $\xi\in E$. Thus, for any $\xi\in E$ with $\norm{\xi}{E}=1$,
\[
\sup_{x\in \varsigma(F)}\norm{x\cdot\xi-\xi\cdot x}{E}=\sup_{a\in F}\norm{a\cdot\xi-\xi\cdot a}{E}\leq\sup_{k\in\{1,\ldots,n\}}\norm{a_k\cdot\xi-\xi\cdot a_k}{E}+2\varepsilon.
\]
So any net of vectors in $E$ which is almost $\mathcal{A}$-central is also strictly almost $\mathcal{A}$-central. Hence, if $E$ admits a net of almost $\mathcal{A}$-central unit vectors, property $(T_{\mathcal{E}})$ of $\mathcal{A}$ will imply the existence of a net $\net{\eta}{i}{I}$ of central vectors such that $\norm{\xi_i-\eta_i}{E}$ converges to zero.
\end{proof}
\begin{rk}
The similar statement to that of Proposition \ref{prop:TE_for_unital_algebras} but for weak property $(\mathrm{T}_{\mathcal{E}})$ also holds with essentially the same proof.
\end{rk}

Before proceeding to specific cases, we record the following permanence property:
\begin{prop}\label{prop:TE_preserved_under_hom_w_dense_im}
Let $\mathcal{A}$ be a Banach algebra with a bounded approximate unit, let $\mathcal{B}$ be another Banach algebra and let $\varphi:\mathcal{A}\rightarrow\mathcal{B}$ be a bounded homomorphism with dense range. If $\mathcal{A}$ has (weak) property $(\mathrm{T}_{\mathcal{E}})$, for a class of Banach spaces $\mathcal{E}$, then so does $\mathcal{B}$.
\end{prop}
\begin{proof}
Let $E\in\mathcal{E}$ be an essential $\mathcal{B}$-bimodule admitting a net $\net{\xi}{i}{I}$ of strictly almost $\mathrm{M}(\mathcal{B})$-central unit vectors. Through precomposition with $\varphi$, $E$ becomes an $\mathcal{A}$-bimodule, and as such, it is  essential because $\varphi$ has dense range. We check that the net $\net{\xi}{i}{I}$ remains almost central for the induced $\mathrm{M}(\mathcal{A})$-bimodule structure. By Theorem \ref{thm:extending_hom_to_DC}, $\varphi$ extends to a strictly continuous homomorphism $\Phi:\mathrm{M}(\mathcal{A})\rightarrow\mathrm{M}(\mathcal{B})$ and the $\mathrm{M}(\mathcal{A})$-bimodule structure on $E$ induced through $\Phi$ from the $\mathrm{M}(\mathcal{B})$-bimodule structure agrees with the extension of the $\mathcal{A}$-bimodule structure induced through $\varphi$ from the $\mathcal{B}$-bimodule structure. Let $S\subset\mathrm{M}(\mathcal{A})$ be any strictly compact subset. The image of $S$ under $\Phi$ is then a strictly compact subset of $\mathrm{M}(\mathcal{B})$. Hence,
\[
\sup_{x\in S}\norm{x\cdot\xi_i-\xi_i\cdot x}{E}=\sup_{x\in S}\norm{\Phi(x)\cdot\xi_i-\xi_i\cdot \Phi(x)}{E}=\sup_{y\in \Phi(S)}\norm{y\cdot\xi_i-\xi_i\cdot y}{E}\rightarrow0,
\]
and so, $\net{\xi}{i}{I}$ is a net of strictly almost $\mathcal{A}$-central unit vectors. Now, if $\xi$ is any $\mathcal{A}$-central vector, then density of the image of $\mathcal{A}$ under $\varphi$ implies that $\xi$ must also be $\mathcal{B}$-central. Hence, if $\mathcal{A}$ has (weak) property $(T_{\mathcal{E}})$ then so does $\mathcal{B}$.
\end{proof}

\subsection{Locally compact groups and their pseudofunction algebras}
In this section we provide a characterization of (weak) property $(\mathrm{T}_{\mathcal{E}})$ of a locally compact group $G$ in terms of (weak) property $(\mathrm{T}_\mathcal E)$ of $F^*_{\mathcal E}(G)$ for a class of Banach spaces $\mathcal E$; see Theorem \ref{thm:(T_E)_for_G_iff_for_FstarE(G)} and \ref{thm:(T_E)_for_G_iff_for_FstarE(G)weak}. 
It generalizes the similar result of Bekka and Ng in Theorem 1 in \cite{BekkaNg2018PropertyTCstar} from the $C^*$-algebra setting to the Banach algebra setting. The generalization comes in two versions reflecting the fact that, unlike in the Hilbert spaces setting, property $(\mathrm{T}_ \mathcal E)$ for a group need not be equivalent to its weak relative. The proof relies on a natural way of constructing an isometric representation from an $F^*_{\mathcal{E}}(G)$-bimodule, and vice versa.\\

Let $\mathcal{E}$ be a class of Banach spaces and let $E\in\mathcal{E}$ be an essential $F^*_{\mathcal{E}}(G)$-bimodule with left and right actions
\[
\varphi:F^*_{\mathcal{E}}(G)\rightarrow \mathscr{B}(E)\qquad\mbox{and}\qquad \psi:F^*_{\mathcal{E}}(G)^{\mathrm{op}}\rightarrow\mathscr{B}(E).
\]
By Proposition \ref{prop:isom_rep_of_G_is_1-to-1_contr_rep_of_L1(G)} and Remark \ref{rk:isom_rep_of_G_is_1-to-1_contr_rep_of_F_R(G)}, $\varphi$ and $\psi$ are induced from isometric representations of $G$, respectively $G^{\mathrm{op}}$, which we shall also denote by $\varphi$ and $\psi$. Since $\varphi$ and $\psi$ have commuting ranges as left and right actions of $F^*_{\mathcal{E}}(G)$, their ranges as group representations commute as well. We construct a new isometric representation $\pi$ of $G$ on $E$ by setting
\[
\pi(t)\xi=\varphi(t)\psi(t^{-1})\xi,\qquad\mbox{for all }t\in G\mbox{ and }\xi\in E. \numberthis\label{eqn:isom_rep_from_bimodule}
\]
It is easy to see that a vector $\xi\in E$ is $G$-invariant if and only if it is $F^*_{\mathcal{E}}(G)$-central. Further, for every vector $\xi\in E$ and every $t\in G$, Proposition \ref{prop:DC_bimodule_structure_encodes_gp_action} yields that
\[
\norm{\pi(t)\xi-\xi}{E}=\norm{\varphi(t)\xi-\psi(t)\xi}{E}=\norm{(L_t,\Delta(t^{-1})R_{t^{-1}})\cdot\xi-\xi\cdot(L_t,\Delta(t^{-1})R_{t^{-1}})}{E},
\]
Because the embedding $G\hookrightarrow\mathrm{M}(F^*_{\mathcal{E}}(G))$ is continuous when $\mathrm{M}(F^*_{\mathcal{E}}(G))$ is equipped with the strict topology, we see that any net of strictly almost $\mathrm{M}(F^*_{\mathcal{E}}(G))$-central vectors is also a net of almost $G$-invariant vectors.\\

Now let $(\pi,E)$ be an isometric representation of $G$ with $E \in \mathcal{E}$. Then $\pi$ extends to a non-degenerate, contractive representation of $F^*_{\mathcal{E}}(G)$ on $E$. Further, since the trivial representation $1_G$ is contained in $\Rep_{\mathcal{E}}(G)$, it extends to $F^*_{\mathcal{E}}(G)$. This induces an essential $F^*_{\mathcal{E}}(G)$-bimodule structure on $E$ with left action $\pi$ and right action $1_G$. It is easy to see that the $F^*_{\mathcal{E}}(G)$-central vectors for this bimodule structure are exactly the $G$-invariant vectors. Further, for each $\xi\in E$ and each $f\in C_c(G)$,

\begin{align*}
\norm{\pi(f)\xi-1_G(f)\xi}{E}
&=\norm{\integral{f(s)\left(\pi(s)\xi-\xi\right)}{G}{\mu_G(s)}}{E}\\
&\leq \integral{\abs{f(s)}\norm{\pi(s)\xi-\xi}{E}}{G}{\mu_G(s)}\\ 
&\leq \sup_{s\in\supp f}\norm{\pi(s)\xi-\xi}{E}\norm{f}{1}.
\end{align*}
Let $x\in F^*_{\mathcal{E}}(G)$. For any $\varepsilon>0$, we can find $f\in C_c(G)$ such that $\norm{x-f}{F^*_{\mathcal{E}}(G)}<\varepsilon$. Then
\begin{align*}
\norm{\pi(x)\xi-1_G(x)\xi}{E}
< \norm{\pi(f)\xi-1_G(f)\xi}{E}+ 2\varepsilon.
\end{align*}
Hence, if $\net{\xi}{i}{I}$ is a net in $E$ of almost invariant unit vectors then it is almost $F^*_{\mathcal{E}}(G)$-central for the bimodule structure on $E$ with left action $\pi$ and right action $1_G$. In fact, as we shall see next, it will be almost central for the extension of the bimodule structure to the multiplier algebra. We show this in the following technical lemma, which is based on the proof of Proposition 10 in \cite{BekkaNg2018PropertyTCstar}.
\begin{lem}\label{lem:bimodule_str_from_rep_almost_central_vectors}
Let $(\pi,E)$ be an isometric representation of the locally compact group $G$ with $E$ in the class $\mathcal{E}$, and view $E$ as an $F^*_{\mathcal{E}}(G)$-bimodule with left action $\pi$ and right action $1_G$. Then any net of almost $F^*_{\mathcal{E}}(G)$-central unit vectors is automatically strictly almost $F^*_{\mathcal{E}}(G)$-central.
\end{lem}
\begin{proof}
Suppose $\net{\xi}{i}{I}$ is a net of almost $F^*_{\mathcal{E}}(G)$-central unit vectors in $E$. Fix $a_0\in F^*_{\mathcal{E}}(G)$ such that $1_G(a_0)=1$. Then
\[
\abs{\norm{\pi(a_0)\xi_i}{E}-1}\leq\norm{\pi(a_0)\xi_i-1_G(a_0)\xi_i}{E},
\]
for all $i\in I$, and so,
\[
\lim_i\norm{\pi(a_0)\xi_i}{E}=1.
\]
We may assume that $\pi(a_0)\xi_i$ is non-zero for all $i\in I$ as we can otherwise pass to a subnet. Define, for each $i\in I$,
\[
\eta_i=\frac{\pi(a_0)\xi_i}{\norm{\pi(a_0)\xi_i}{E}}.
\]
We claim that $\net{\eta}{i}{I}$ constitutes a net of strictly almost $F^*_{\mathcal{E}}(G)$-central unit vectors. To see this, let $S$ be any strictly compact subset of $\mathrm{M}(F^*_{\mathcal{E}}(G))$. Given $\varepsilon>0$, we can find a finite collection of elements $x_1,\ldots,x_n$ of $S$ such that, for every $x\in S$, there is a $k\in\{1,\ldots,n\}$ for which
\[
\norm{xa_0-x_ka_0}{F^*_{\mathcal{E}}(G)}<\frac{\varepsilon}{12}.\numberthis\label{eqn:suff_cond_for_containing_1-dim_rep_of_B-alg_1}
\]
Take $i_0\in I$ such that the following hold, for all $i\succcurlyeq i_0$ and all $k=1,\ldots,n$,
\begin{gather*}
\norm{\pi(a_0)\xi_i}{E}\geq\frac{1}{2},\\ 
\norm{\pi(a_0)\xi_i-1_G(a_0)\xi_i}{E}<\frac{\varepsilon}{4\sup_{y\in S}{\norm{y}{}}},\\ 
\norm{\pi(x_ka_0)\xi_i-1_G(x_ka_0)\xi_i}{E}<\frac{\varepsilon}{12}.
\end{gather*}
Now, given $x\in S$, take $k\in\{1,\ldots,n\}$ such that (\ref{eqn:suff_cond_for_containing_1-dim_rep_of_B-alg_1}) holds. Then
\begin{align*}
\norm{x\cdot\eta_i-\eta_i\cdot x}{E}&=\frac{1}{\norm{\pi(a_0)\xi_i}{E}}\norm{\pi(x)\pi(a_0)\xi_i-1_G(x)\pi(a_0)\xi_i}{E}\\ 
&\leq2\norm{\pi(xa_0)\xi_i-\pi(x_ka_0)\xi_i}{E}+2\norm{\pi(x_ka_0)\xi_i-1_G(x_ka_0)\xi_i}{E}\\ 
&\phantom{=}+2\norm{1_G(x_ka_0)\xi_i-1_G(xa_0)\xi_i}{E}+2\norm{1_G(xa_0)\xi_i-1_G(x)\pi(a_0)\xi_i}{E}\\ 
&\leq \frac{\varepsilon}{6}+\frac{\varepsilon}{6}+\frac{\varepsilon}{6}+\abs{1_G(x)}\frac{\varepsilon}{2\sup_{y\in S}{\norm{y}{}}} <\varepsilon.
\end{align*}
Thus, $\net{\eta}{i}{I}$ is indeed a net of strictly almost $F^*_{\mathcal{E}}(G)$-central unit vectors in $E$. Now, by construction of the net $\net{\eta}{i}{I}$, the norm difference $\norm{\xi_i-\eta_i}{E}$ converges to zero. Hence, $\net{\xi}{i}{I}$ is a net of strictly almost $F^*_{\mathcal{E}}(G)$-central unit vectors in $E$, as well.
\end{proof}

Theorem \ref{thm:(T_E)_for_G_iff_for_FstarE(G)} below, which is one of our main results, relates property $(\mathrm{T}_{\mathcal{E}})$ for a locally compact group $G$ with property $(\mathrm{T}_{\mathcal{E}})$ of its associated symmetrised $\mathcal{E}$-pseudofunction algebra.

\refstepcounter{thm}\label{thm:(T_E)_for_G_iff_for_FstarE(G)bothversions}
\begin{subthm}\label{thm:(T_E)_for_G_iff_for_FstarE(G)}
Let $G$ be a second countable locally compact group and let $\mathcal{E}$ be a class of Banach spaces. The following are equivalent:
\begin{enumerate}[(i)]
\item $G$ has property $(\mathrm{T}_{\mathcal{E}})$,
\item $F^*_{\mathcal{E}}(G)$ has property $(\mathrm{T}_{\mathcal{E}})$.
\end{enumerate}
\end{subthm}
\begin{proof}
$(i)\Rightarrow(ii)$: Assume that $G$ has property $(\mathrm{T}_{\mathcal{E}})$ and let $E\in\mathcal{E}$ be an essential $F^*_{\mathcal{E}}(G)$-bimodule admitting a net $\net{\xi}{i}{I}$ of strictly almost $F^*_{\mathcal{E}}(G)$-central unit vectors. Then $\net{\xi}{i}{I}$ is almost $G$-invariant for the isometric representation $\pi$ of $G$ induced by the $F^*_{\mathcal{E}}(G)$-bimodule structure. By the assumption that $G$ has property $(\mathrm{T}_{\mathcal{E}})$, we obtain a net of $G$-invariant vectors $\net{\eta}{i}{I}$ such that $\norm{\xi_i-\eta_i}{E}\rightarrow0$. Since the $C_c(G)$ is dense in $F^*_{\mathcal{E}}(G)$, we see that each $\eta_i$ is $F^*_{\mathcal{E}}(G)$-central. Thus, $F^*_{\mathcal{E}}(G)$ has property $(\mathrm{T}_{\mathcal{E}})$.

$(ii)\Rightarrow(i)$: Assume $F^*_{\mathcal{E}}(G)$ has property $(\mathrm{T}_{\mathcal{E}})$ and let $(\pi,E)$ be an isometric representation of $G$ on a Banach space $E$ in $\mathcal{E}$ admitting a net $\net{\xi}{i}{I}$ of almost invariant unit vectors. Then $\net{\xi}{i}{I}$ is almost $F^*_{\mathcal{E}}(G)$-central for the bimodule structure on $E$ with left action $\pi$ and right action $1_G$. Hence, $\net{\xi}{i}{I}$ is automatically strictly almost $F^*_{\mathcal{E}}(G)$-central, by Lemma \ref{lem:bimodule_str_from_rep_almost_central_vectors}. By the assumption that $F^*_{\mathcal{E}}(G)$ has property $(\mathrm{T}_{\mathcal{E}})$, we obtain a net $\net{\eta}{i}{I}$ of $F^*_{\mathcal{E}}(G)$-central vectors such that $\norm{\xi_i-\eta_i}{E}\rightarrow0$. Hence, $G$ has property $(\mathrm{T}_{\mathcal{E}})$.
\end{proof}

\begin{rk}
Theorem \ref{thm:(T_E)_for_G_iff_for_FstarE(G)} also holds with $F_{\mathcal{E}}(G)$ in place of $F^*_{\mathcal{E}}(G)$. In fact, the proof shows the following stronger statement: If $G$ has property $(\mathrm{T}_{\mathcal{E}})$ then $F_{\mathcal{R}}(G)$ has property $(\mathrm{T}_{\mathcal{E}})$ for any class $\mathcal{R}$ of isometric Banach space representations of $G$. Further, a sufficient condition for the converse implication is that $\mathcal{R}$ contains the class of all isometric representations of $G$ on a space in $\mathcal{E}$.
\end{rk}

The similar statement holds when exchanging property $(\mathrm{T}_{\mathcal{E}})$ with weak property $(\mathrm{T}_{\mathcal{E}})$. The proof is the same mutatis mutantis, and so, we omit it.
\addtocounter{thm}{-1}
\begin{subthm}\label{thm:(T_E)_for_G_iff_for_FstarE(G)weak}
Let $G$ be a locally compact group and $\mathcal{E}$ a class of Banach spaces. The following are equivalent:
\begin{enumerate}[(i)]
\item $G$ has weak property $(\mathrm{T}_{\mathcal{E}})$,
\item $F^*_{\mathcal{E}}(G)$ has weak property $(\mathrm{T}_{\mathcal{E}})$.
\end{enumerate}
\end{subthm}
\setcounter{subthm}{0}
\addtocounter{thm}{1}

\begin{rk}
When $\mathcal{E}$ is a class satisfying either of the two conditions of Proposition \ref{prop:on_eq_of_(T_E)_and_weak(T_E)} so that property $(\mathrm{T}_{\mathcal{E}})$ and weak property $(\mathrm{T}_{\mathcal{E}})$ for the group are equivalent, the two theorems \ref{thm:(T_E)_for_G_iff_for_FstarE(G)} and \ref{thm:(T_E)_for_G_iff_for_FstarE(G)weak} can be merged into one. This holds, in particular, when $\mathcal{E}$ is the class of complex Hilbert spaces, in which case we recover the similar result of Bekka and Ng in Theorem 1 in \cite{BekkaNg2018PropertyTCstar}.
\end{rk}

As an immediate corollary to Theorem \ref{thm:(T_E)_for_G_iff_for_FstarE(G)} and to Theorem A in \cite{BaderFurmanGelanderMonod}, we obtain the following equivalence:

\begin{cor} \label{cor:(T_Lp)_iff_(T_Lq)}
Let $G$ be a second countable locally compact group and let $1\leq p,q<\infty$. Then $F^*_{L^p}(G)$ has property $(\mathrm{T}_{L^p})$ if and only if $F^*_{L^q}(G)$ has property $(\mathrm{T}_{L^q})$.
\end{cor}

\subsection{Property $(\mathrm{T}_{L^q})$ for $F^*_{L^p}(G)$}
Let $\mathcal{E}$ be the class of $L^p$-spaces. One may view the associated Banach $^*$-algebras $F^*_{L^p}(G)$ as interpolating between $L^1(G)$ and the universal group $C^*$-algebra $C^*(G)$ as $p$ varies from $1$ to $2$. Precisely, for $1\leq q\leq p\leq 2$, the identity on $L^1(G)$ extends to a contractive homomorphism $F^*_{L^q}(G)\rightarrow F^*_{L^p}(G)$.\footnote{ 
This follows from the similar result in the classical (non-symmetrized) setting by Gardella and Thiel (see Theorem 2.30 in \cite{GardellaThiel2014GpAlgActingOnLp}), but it can also be proven more directly in the symmetrized setting via interpolation theory. 
The latter proof is part of work in progress of the first named author.} With this in mind, we obtain the following results as consequences to Theorem \ref{thm:(T_E)_for_G_iff_for_FstarE(G)}: 

\begin{cor}\label{cor:FLp_has_TLq_lcgp}
Let $G$ be a second countable locally compact group with property $(\mathrm{T})$ and let $1\leq p\leq2$. Then $F^*_{L^p}(G)$ has property $(\mathrm{T}_{L^q})$, for all $1\leq q\leq p$ and all $p'\leq q<\infty$, where $p'$ is the Hölder conjugate of $p$.
\end{cor}
\begin{proof}
Since $G$ has property $(\rm T)$, it has property $(\mathrm{T}_{L^q})$, for every $1\leq q<\infty$, by Theorem A(i) in \cite{BaderFurmanGelanderMonod}, and so, $F^*_{L^q}(G)$ has property $(\mathrm{T}_{L^q})$, by Theorem \ref{thm:(T_E)_for_G_iff_for_FstarE(G)}. For $1\leq q\leq p$ or $p'\leq q<\infty$, the identity on $L^1(G)$ extends to a contractive homomorphism $F^*_{L^q}(G)\rightarrow F^*_{L^p}(G)$ with dense range. It follows by Proposition \ref{prop:TE_preserved_under_hom_w_dense_im} that $F^*_{L^p}(G)$ has property $(\mathrm{T}_{L^q})$.
\end{proof}

For discrete groups we obtain a similar result also for parameters $q$ in the interval between $p$ and $p'$.

\begin{cor}\label{cor:FLp_has_TLq_disgp}
Let $\Gamma$ be a discrete group with property $(\rm T)$ and let $1\leq p\leq2$. Then $F^*_{L^p}(\Gamma)$ has property $(\mathrm{T}_{L^q})$, for all $1\leq q<\infty$.
\end{cor}
\begin{proof}
It suffices to show the statement for $p< q< p'$; the cases where $q$ is in between $1$ and $p$ or greater than $p'$ are covered in Corollary \ref{cor:FLp_has_TLq_lcgp}. As in the proof of Corollary \ref{cor:FLp_has_TLq_lcgp}, it follows from Theorem A(i) \cite{BaderFurmanGelanderMonod} and Theorem \ref{thm:(T_E)_for_G_iff_for_FstarE(G)} that $F^*_{L^q}(\Gamma)$ has property $(\mathrm{T}_{L^q})$. Let $L^q(\Omega,\nu)$ be an $F^*_{L^p}(\Gamma)$-bimodule admitting a net $\net{\xi}{i}{I}$ of almost $F^*_{L^p}(\Gamma)$-central unit vectors. By construction of $F^*_{L^q}(\Gamma)$, we see that the $F^*_{L^p}(\Gamma)$-bimodule actions extend continuously to $F^*_{L^q}(\Gamma)$. Let $\{x_1,\ldots,x_n\}$ be any finite subset of $F^*_{L^q}(\Gamma)$ and let $\varepsilon>0$. Take $f_1,\ldots, f_n\in \ell^1(\Gamma)$ such that $\norm{x_j-f_j}{F^*_{L^q}(\Gamma)}<\varepsilon$, for $j=1,\ldots, n$. Then,
\[
\sup_{j\in\{1,\ldots,n\}}\norm{x_j\cdot\xi_i-\xi_i\cdot x_j}{L^q(\Omega,\nu)}<\sup_{j\in\{1,\ldots,n\}}\norm{f_j\cdot\xi_i-\xi_i\cdot f_j}{L^q(\Omega,\nu)}+2\varepsilon.
\]
Because we can view $\{f_1,\ldots,f_n\}$ as a finite subset of $F^*_{L^p}(\Gamma)$, the supremum on the right-hand side can be made arbitrarily small when $i$ is chosen large enough. It follows that $\net{\xi}{i}{I}$ is almost central for the $F^*_{L^q}(\Gamma)$-bimodule structure. By Proposition \ref{prop:TE_for_unital_algebras}, property $(\mathrm{T}_{L^q})$ for $F^*_{L^q}(\Gamma)$ then implies the existence of a net $\net{\eta}{i}{I}$ in $L^q(\Omega,\nu)$ consisting of $F^*_{L^q}(\Gamma)$-central vectors such that $\norm{\xi_i-\eta_i}{L^q(\Omega,\nu)}$ converges to zero. As the $F^*_{L^p}(\Gamma)$-bimodule actions are the precomposition of the $F^*_{L^q}(G)$-bimodule actions with the canonical contractive homomorphism $F^*_{L^p}(\Gamma)\rightarrow F^*_{L^q}(\Gamma)$, the net $\net{\eta}{i}{I}$ is also central for the $F^*_{L^p}(\Gamma)$-bimodule structure. We conclude from Proposition \ref{prop:TE_for_unital_algebras} that $F^*_{L^p}(\Gamma)$ has property $(\mathrm{T}_{L^q})$.
\end{proof}

\section{Property $(\mathrm{T}_{L^p})$ for symmetrized $p$-pseudofunction algebras}
In this section, we continue to focus our attention to the class of $L^p$-spaces on $\sigma$-finite measure spaces, where $1\leq p<\infty$. 
Due to Theorem \ref{thm:(T_E)_for_G_iff_for_FstarE(G)}, $G$ has property $(\mathrm{T}_{L^p})$ if and only if $F^*_{L^p}(G)$ has property $(\mathrm{T}_{L^p})$. For $p=2$, we recover the result of Bekka and Ng in Theorem 1 in \cite{BekkaNg2018PropertyTCstar} that $G$ has property $(\rm T)$ if and only if the universal group $C^*$-algebra $C^*(G)$ of $G$ has property $(\rm T)$. 
After establishing their result, Bekka and  Ng ask if $C^*(G)$ can be replaced by the reduced group $C^*$-algebra $C^*_r(G)$ of $G$. 
In this section, we ask the same question but in the more general setting of actions on $L^p$-spaces. 
Here, the role of the reduced group $C^*$-algebra is played by $F^*_{\lambda_p}(G)$. Thus, we ask if (or when) property $(\mathrm{T}_{L^p})$ for the group is captured by $F^*_{\lambda_p}(G)$. We shall see that this is the case when $G$ is a discrete group (see Theorem \ref{TLp_for_IN_gps}). This result generalizes that of Bekka and Ng in the case of discrete groups. Our proof relies on a generalisation of Fell's absorption principle to isometric representations on $L^p$-spaces. Further, when $G$ is discrete, we show that weak property $(\mathrm{T}_{SL^p})$ for $F^*_{\lambda_p}(G)$ implies property $(\mathrm{T}_{L^p})$ for the group (see Theorem \ref{thm:TLp_for_discrete_gps}), where $SL^p$ is the class of closed subspaces of $L^p$-spaces on $\sigma$-finite measure spaces.\\

Let $(\pi,L^p(\Omega,\nu))$ be an isometric $L^p$-representation of the locally compact group $G$. We denote by $\id$ the trivial representation of $G$ on $L^p(\Omega,\nu)$ and by $\lambda_p$ the left-regular representation of $G$ on $L^p(G)$. Consider the $L^p$-space
\[
E= L^p(G,L^p(\Omega,\nu)).\numberthis\label{eqn:Fpstar-bimodule_from_iso_rep}
\]
This space contains the algebraic tensor product $L^p(G)\odot L^p(\Omega,\nu)$ as a dense subspace. Hence, $\lambda_p\otimes\pi$ and $\lambda_p\otimes\id$ defines isometric representations of $G$ on $E$. For $p=2$, we know from Fell's absorption principle that these two representations are unitarily equivalent. For general $1\leq p<\infty$, an $L^p$-version of the Fell's absorption principle was shown in Proposition 5.1 \cite{Runde2004}. We state it in Proposition \ref{prop:LpFellsAbsorption} below in the form we need it.

\begin{prop}[$L^p$-version of Fell's absorption principle]\label{prop:LpFellsAbsorption}
Let $G$ be a locally compact group, let $1\leq p<\infty$, and let $(\pi,L^p(\Omega,\nu))$ be an isometric representation of $G$. Then $\lambda_p\otimes\pi$ and $\lambda_p\otimes\id$ are equivalent in the sense that they are intertwined by a surjective isometry of $L^p(G;L^p(\Omega,\nu))$.
\end{prop}

\begin{rk}
In Proposition \ref{prop:LpFellsAbsorption}, one can exchange the left regular representation with the right regular representation, $\rho_p$. That is, $\rho_p\otimes\pi$ is equivalent to $\rho_p\otimes\id$. The proof is the same mutatis mutantis.
\end{rk}

Using the $L^p$-version of Fell's absorption principle, we can construct an $F^*_{\lambda_p}(G)$-bimodule on $E$ from the isometric representation $(\pi,L^p(\Omega,\nu))$ of $G$ as follows: Set $\varphi=\lambda_p\otimes\id$ and $\psi=\rho_p\otimes\pi$ of $G$ on $E$. Clearly, $\varphi$ integrates to a representation of $F^*_{\lambda_p}(G)$, and $\psi$ does as well by Lemma \ref{prop:LpFellsAbsorption} and the remark following it. Thus, $E$ is an $F^*_{\lambda_p}(G)$-bimodule with left action $\varphi$ and right action $\psi^{\mathrm{op}}=\psi\circ\widetilde{\square}$.

\begin{lem}\label{lem:central_vectors_in_Fpstar_bimodule_from_iso_rep}
Let $G$ be a locally compact group, let $(\pi,L^p(\Omega,\nu))$ be an isometric representation and let $E$ be the $F^*_{\lambda_p}(G)$-bimodule from equation (\ref{eqn:Fpstar-bimodule_from_iso_rep}). If $\eta\in E$ is central then
\[
\pi(s)\eta(t)=\eta(sts^{-1}),
\]
for all $s\in G$ and $\mu_G$-almost all $t\in G$.
\end{lem}
\begin{proof}
Let $\eta\in E$ be central. Then $\eta$ is also central for the extension of the bimodule structure to $\mathrm{M}(F^*_{\lambda_p}(G))$. Hence, for every $s\in G$, $\psi(s)\eta=\psi^{\mathrm{op}}(s^{-1})\eta=\varphi(s^{-1})\eta$, where the equality is in $E$. It follows that
\[
\pi(s)\eta(t)=\psi(s)\eta(ts^{-1})=\varphi(s^{-1})\eta(ts^{-1})=\eta(sts^{-1}),
\]
for each $s\in G$ and for $\mu_G$-almost every $t\in G$.
\end{proof} 

\subsection{Property $(\mathrm{T}_{L^p})$ for $F^*_{\lambda_p}$ for discrete groups}
We show next that property $(\mathrm{T}_{L^p})$ for a discrete group $\Gamma$ is detected by its (symmetrized) $p$-pseudofunction algebra. Recall that pseudofunction algebras of discrete groups are unital, and we can therefore use the equivalent definition of property $(\mathrm{T}_{L^p})$ from Proposition \ref{prop:TE_for_unital_algebras}. 

\begin{thm}\label{TLp_for_IN_gps}
Let $\Gamma$ be a discrete group. For each $1\leq p<\infty$, the following are equivalent:
\begin{enumerate}[(i)]
\item $\Gamma$ has property $(\mathrm{T}_{L^p})$,
\item $F^*_{L^p}(\Gamma)$ has property $(\mathrm{T}_{L^p})$,
\item $F^*_{\lambda_p}(\Gamma)$ has property $(\mathrm{T}_{L^p})$.
\end{enumerate}
\end{thm}
Our proof uses the ideas of the proof of Theorem 9 in \cite{BekkaNg2018PropertyTCstar}.

\begin{proof}
$(i)\Rightarrow(ii)$ is covered by Theorem \ref{thm:(T_E)_for_G_iff_for_FstarE(G)} and $(ii)\Rightarrow(iii)$ follows from Proposition \ref{prop:TE_preserved_under_hom_w_dense_im}. It remains to show $(iii)\Rightarrow(i)$. Suppose $(\pi,L^p(\Omega,\nu))$ is an isometric representation of $\Gamma$ and let $E$ be the $F^*_{\lambda_p}(\Gamma)$-bimodule from equation (\ref{eqn:Fpstar-bimodule_from_iso_rep}). Given $\xi\in L^p(\Omega,\nu)$ set $\zeta=\delta_e\otimes\xi\in E$. For each $f\in C_c(\Gamma)$, we have
\begin{align*}
f\cdot\zeta&=\varphi(f)(\zeta)
=\sum_{r\in\Gamma}{f(r)\delta_r\otimes\xi},\\
\zeta\cdot f&=\psi^{\mathrm{op}}(f)(\zeta)
=\sum_{r\in\Gamma}{f(r)\delta_r\otimes\pi(r^{-1})\xi}.
\end{align*}
We compute
\begin{align*}
\norm{f\cdot\zeta-\zeta\cdot f}{E}
&=\norm{\sum_{r\in\Gamma}{f(r)\delta_r\otimes(\xi-\pi(r^{-1})\xi)}}{E}\\
&=\left(\sum_{s\in\Gamma}{\norm{f(s)(\xi-\pi(s^{-1})\xi)}{L^p(\Omega,\nu)}^p}\right)^{1/p}\\
%&\leq \norm{f}{p}\sup_{t\in\supp(f)}\norm{\pi(t)\xi-\xi}{L^p(\Omega,\nu)}\\ 
&\leq \norm{f}{1}\sup_{t\in\supp(f)}\norm{\pi(t)\xi-\xi}{L^p(\Omega,\nu)}.
\end{align*}
Suppose $\pi$ admits a net $\net{\xi}{i}{I}$ of almost $\Gamma$-invariant unit vectors in $L^p(\Omega,\nu)$. For each $i\in I$, set $\zeta_i=\delta_e\otimes\xi_i$. Given $f\in C_c(\Gamma)$ and $\varepsilon>0$, pick $i_{f,\varepsilon}\in I$ such that, for all $i\succcurlyeq i_{f,\varepsilon}$,
\[
\sup_{t\in\supp(f)}\norm{\pi(t)\xi_i-\xi_i}{L^p(\Omega,\nu)}<\varepsilon/\norm{f}{1}.
\]
By the above calculations, we see that $\norm{f\cdot\zeta_i-\zeta_i\cdot f}{E}<\varepsilon$, for all $i\succcurlyeq i_{f,\varepsilon}$. Since $C_c(\Gamma)$ is dense in $F^*_{\lambda_p}(\Gamma)$, we deduce that $\net{\zeta}{i}{I}$ is an almost $F^*_{\lambda_p}(\Gamma)$-central net. By the assumption that $F^*_{\lambda_p}(\Gamma)$ has property $(\mathrm{T}_{L^p})$, we obtain a net $\net{\eta}{i}{I}$ of $F^*_{\lambda_p}(\Gamma)$-central vectors satisfying
\[
\norm{\delta_e\otimes\xi_i-\eta_i}{L^p(G;L^p(\Omega,\nu))}\rightarrow0,
\]
By Lemma \ref{lem:central_vectors_in_Fpstar_bimodule_from_iso_rep}, we have equality $\pi(s)(\eta_i(t))=\eta_i(s^{-1}ts)$, for all $s,t\in G$ and every $i\in I$. Hence, $\eta_i(e)$ is a $\Gamma$-invariant vector in $L^p(\Omega,\nu)$. Further,
\[
\norm{\xi_i-\eta_i(e)}{L^p(\Omega,\nu)}\leq\norm{\delta_e\otimes\xi_i-\eta_i}{E}\rightarrow 0.
\]
Hence, we conclude that $\Gamma$ has property $(\mathrm{T}_{L^p})$.
\end{proof}

\begin{rk}
Theorem \ref{TLp_for_IN_gps} also holds with $F_{L^p}(\Gamma)$ and $F_{\lambda_p}(\Gamma)$ in place of $F^*_{L^p}(\Gamma)$ and $F^*_{\lambda_p}(\Gamma)$. The proof is the same.
\end{rk}

\begin{rk}
%For $p=2$, Theorem \ref{TLp_for_IN_gps} recovers Theorem 9(c) in \cite{BekkaNg2018PropertyTCstar} for discrete groups. 
The implications (i)$\Rightarrow$(ii) and (ii)$\Rightarrow$(iii) in Theorem \ref{TLp_for_IN_gps} hold for general second countable locally compact groups. For $p=2$, it is shown in Example 6 in \cite{BekkaNg2018PropertyTCstar} that the implication (iii)$\Rightarrow$(i) fails for general locally compact groups, but it holds for IN groups (see Theorem 9(c) in \cite{BekkaNg2018PropertyTCstar}) -- a class of locally compact groups which contains, but is strictly larger than, the class of discrete groups. When $p\neq 2$ and $G$ is a locally compact non-discrete IN group, the proof given above with the appropriate adjustments (analogue to \cite{BekkaNg2018PropertyTCstar}) shows that property $(\mathrm{T}_{L^p})$ for $F^*_{\lambda_p}(G)$ implies weak property $(\mathrm{T}_{L^p})$ for $G$. 
For countable discrete groups the equivalence between weak property $(\mathrm{T}_{L^p})$ and property $(\mathrm{T}_{L^p})$ is known (see Theorem C in \cite{elkiaer2024weakTLp}), however it is open whether this is true for general locally compact groups.
\end{rk}

As a corollary to Theorem \ref{TLp_for_IN_gps}, we establish a relation between property $(\mathrm{T}_{L^p})$ and amenability on the level of the $p$-pseudofunction algebra. The first part of the proof is analogous to the proof given in \cite{GardellaThiel2014GpAlgActingOnLp} of their Theorem 3.11.

\begin{cor}
	Let $\Gamma$ be a discrete group. Suppose $\mathrm F^*_{\lambda_p}(\Gamma)$ has property 
	$(\mathrm T_{L^p})$ and it is amenable as a Banach algebra. Then it is finite dimensional. 
\end{cor}

\begin{proof}
	By Theorem 2.3.1 in \cite{Runde2020AmenableBA}, amenability of $\mathrm F^*_{\lambda_p}(\Gamma)$ implies that of the reduced group $C^*$-algebra $\mathrm C_r^*(\Gamma)$, which in turn yields amenability of $\Gamma$, by Theorem 2.6.8 in \cite{BrownOsawa}. Moreover, 
    Theorem \ref{TLp_for_IN_gps} combined with Remark \ref{rk: T implies TLp} imply that  $\Gamma$ has Kazhdan property $(\rm T)$. It is well-known that amenability and property $(\rm T)$ of discrete groups imply finiteness. Therefore    
    $\mathrm F^*_{\lambda_p}(\Gamma)=\mathbb C \Gamma$ is finite dimensional.
\end{proof}

\subsection{Weak property ${(\mathrm{T}_{SL^p})}$ for discrete groups}

Denote by $SL^p$ the class of closed subspaces of $L^p$-spaces on $\sigma$-finite measure spaces, for some fixed $1\leq p<\infty$. We show that, for a discrete group $\Gamma$, weak property $(\mathrm{T}_{SL^p})$ implies property $(\mathrm{T}_{L^p})$. Moreover, we show that weak property $(\mathrm{T}_{SL^p})$ for $F^*_{\lambda_p}(\Gamma)$ is intermediate to the two by adapting the proof of Bekka in \cite{Bekka2005PropertyTCstar} that property $(\rm T)$ for a discrete group $\Gamma$ is implied by property $(\rm T)$ for $C^*_r(\Gamma)$.
Along the way, we show that property $(\mathrm{T}_{L^p})$ for $\Gamma$ is implied by the property that isometric representations on spaces in $SL^p$ with almost $\Gamma$-invariant vectors necessarily must have a finite dimensional subrepresentation.
This should be compared with the similar well-known characterization of property $(\rm T)$ in the setting of unitary representations (see Theorem 1 in \cite{BekkaValette1993(T)}).
The original proof of in the setting of unitary representations via the characterization of Kazhdan's property $(\rm T)$ of Delorme and Guichardet utilizes Schönberg's theorem and a GNS-construction to construct a unitary representation with almost invariant vectors. Since we are concerned with isometric representations on $L^p$-spaces, this route does not seem feasible to us. In order to circumvent this, we will provide an alternative proof based on notions from ergodic theory, which employs an idea used in \cite{BaderFurmanGelanderMonod}.  

\begin{defn}
    Let $\Gamma$ be a discrete group and let $1\leq p<\infty$. An isometric representation of $\Gamma$, $(\pi,L^p(\Omega,\nu))$, is called \textit{weakly mixing} if, for each pair of finite sets $E\subset L^p(\Omega,\nu)$ and $F\subset L^q(\Omega,\nu)$ and each $\varepsilon >0$, there exists $t\in \Gamma$ such that 
	\[
	   |\langle{\pi (t) \xi, \eta}\rangle|  < \varepsilon
	\]
	for all $\xi\in E$ and $\eta \in F$.
\end{defn} 

A \emph{probability measure preserving action} (in short, a p.m.p. action) of a discrete group $\Gamma$ on a probability space $(\Omega,\nu)$ is a group homomorphism from $\Gamma$ to the group of bi-measurable transformations of $\Omega$ such that $\nu(t.A)=\nu(A)$, for all $t\in\Gamma$ and all $A\subset\Omega$ measurable.

\begin{defn}
    Let $\Gamma$ be an infinite discrete group and let $\Gamma \curvearrowright (\Omega, \mathcal B, \nu)$ be a p.m.p. action of $\Gamma$ on a probability space $(\Omega, \mathcal B, \nu)$. This action is called \textit{weakly mixing} if, for all $\mathcal F\subset \mathcal B$ finite, we have 
	\[
	  \liminf_{g \to \infty} \sum_{A, B\in \mathcal F}|\nu(A\cap gB) - \nu(A)\nu(B)|=0	
	\]
\end{defn}

For the sake of completeness we include the proof of the following statements.

\begin{prop} \label{prop: weak-mixing action and Koopman rep}
	Let  $\Gamma\curvearrowright (\Omega, \mathcal B, \nu)$ be a p.m.p action and fix $1\leq p<\infty$. If the action is weakly mixing then the Koopman representation 
	$\pi_0 \colon \Gamma \rightarrow \mathrm {Iso}(L^p_0(\Omega, \mathcal B, \nu ))$ is weakly mixing.
\end{prop}

\begin{proof}
	Let $E$ and $F$ be finite sets of simple functions $\sum c_{A_i} \chi_{A_i}$ given by a finite set $\{A_i \in \mathcal B \mid 1\leq i\leq n\}$ and a finite set of coefficients $\{c_{A_i} \mid 1\leq i\leq n\}$, such that $\sum c_{A_i} \nu(A_i)=0$. Let $\mathcal B_0 \subset \mathcal B$ to be the set of all these finitely many measurable sets. Since the action is weakly mixing, there is a sequence $(t_n)_n$ such that for any pair $A, B \in \mathcal B_0$ we have
	\[
	   \nu(t_n A \cap B)\to \nu(A)\nu(B).
	\]
	By bilinearity of the scalar product, we have, for any pair of simple functions in $E$ and $F$,
	\[
	   \left\langle{
	  	\pi_0(t_n) \sum_i c_{A_i} \chi_{A_i}, \sum_j d_{B_i} \chi_{B_i}
  	   }\right\rangle
      = 
      \sum_{i,j}c_{A_i} \overline{d_{B_j}}
      \left\langle {
      \chi_{t_n A_i}, \chi_{B_j}
      }\right\rangle
      \to 
      \sum_{i,j}c_{A_i} \overline{d_{B_j}}\nu(A_i)\nu(B_j)=0 
	\]
	Because simple functions are dense in $L^p_0(\Omega, \mathcal B, \nu )$ and  $L^{p'}_0(\Omega, \mathcal B, \nu )$, this shows that the Koopman representation is weakly mixing.
\end{proof}

\begin{lem} \label{lem: weak-mixing, no f.d. subrep}
	Let $\pi$ be an isometric representation of $\Gamma$ on a closed subspace $X$ of a Banach space. If $\pi$ is weakly mixing, then it does not have any non-zero finite dimensional subrepresentations.
\end{lem}

\begin{proof}
Let $V\subset X$ be a finite dimensional invariant subspace with basis $\{\xi_1,\ldots,\xi_d\}$. 
Because $\pi$ is weakly mixing there exists a sequence $(t_n)_n$ in $\Gamma$ such that $\langle{\pi(t_n) \xi_i, \eta}\rangle \to 0$, for all $i\in\{1,\ldots,d\}$ and all $\eta\in V^*$. Since $V$ is finite dimensional, we may pass to a subsequence of $(t_n)_n$ and assume that there is a $T\in \mathrm{Iso}(V)$ such that $\pi(t_n)\xi \to T\xi$, for all $\xi\in V$. Fix $\xi\in V$. For each $\eta\in V^*$, we find
	\[
	  \langle{T\xi, \eta}\rangle = \langle{\lim \pi(t_n)\xi, \eta}\rangle =0
	\]
Therefore $T\xi=0$. As $T$ is an isometry, this implies $\xi=0$, and so, $V$ must be zero.
\end{proof}

Connes and Weiss provide in \cite{ConnesWeiss1980PropertyTandassympinvsequences} a dynamical characterization of property $(\rm T)$ in terms of ergodic p.m.p. actions which we recall here: A discrete group $\Gamma$ has Kazhdan's property $(\rm T)$ if and only if every p.m.p. ergodic (even weakly mixing) action of $\Gamma$ is strongly ergodic. (See also Theorem 6.3.4 in \cite{BekkaDeLaHarpeValette}).

We are now ready to prove the main result of this subsection.
\begin{thm}\label{thm:TLp_for_discrete_gps}
Let $\Gamma$ be a discrete group, and let $1\leq p<\infty$. Each of the following implies the next:
\begin{enumerate}[(i)]
\item $\Gamma$ has weak property $(\mathrm{T}_{SL^p})$,
\item $F^*_{\lambda_p}(\Gamma)$ has weak property $(\mathrm{T}_{SL^p})$,
\item If an isometric representation $\pi$ of $\Gamma$ on a subspace of some $L^p(\Omega, \nu)$ contains almost $\Gamma$-invariant vectors, then it has a finite dimensional subrepresentation,
\item $\Gamma$ has property $(\mathrm{T}_{L^p})$.
\end{enumerate}
\end{thm}
\begin{proof}
$(i)\Rightarrow(ii)$ follows from Theorem \ref{thm:(T_E)_for_G_iff_for_FstarE(G)weak} and Proposition \ref{prop:TE_preserved_under_hom_w_dense_im}.

$(ii)\Rightarrow(iii)$: Suppose $(\pi,X)$ is an isometric representation of $\Gamma$ with $X\subset L^p(\Omega,\nu)$ a closed subspace. Set
\[
E=\ell^p(\Gamma)\otimes_p X\cong \ell^p(\Gamma,X)\subset \ell^p(\Gamma,L^p(\Omega,\nu)).
\]
Then $E$ is an $F^*_{\lambda_p}(\Gamma)$-bimodule with left action $\varphi=\lambda_p\otimes\id$ and right action $\psi^{\mathrm{op}}=\rho_p^{\mathrm{op}}\otimes\pi^{\mathrm{op}}$.
Given $\xi\in L^p(\Omega,\nu)$ consider the vector $\zeta=\delta_e\otimes\xi$ in $E$. 
We compute
\begin{align*}
\norm{f\cdot\zeta-\zeta\cdot f}{E}&=\norm{\varphi(f)(\delta_e\otimes\xi)-\psi^{\mathrm{op}}(f)(\delta_e\otimes\xi)}{E}\\
&=\norm{\sum_{s\in\Gamma}f(s)\delta_{s}\otimes(\xi-\pi(s^{-1})\xi)}{E}\\ 
&=\left(\sum_{s\in\Gamma}\norm{f(r)(\xi-\pi(r^{-1})\xi)}{L^p(\Omega,\nu)}^p\right)^{1/p}\\ 
&\leq \norm{f}{p}\sup_{r\in\supp(f)}\norm{\pi(r)\xi-\xi}{L^p(\Omega,\nu)}
\end{align*}

Assume $F^*_{\lambda_p}(\Gamma)$ has weak property $(T_{SL^p})$ and suppose $\pi$ admits a net $\net{\xi}{i}{I}$ of almost $\Gamma$-invariant unit vectors in $X$. For each $i\in I$, set $\zeta_i=\delta_e\otimes\xi_i$. Then $\net{\zeta}{i}{I}$ is a net of unit vectors in $E$. By our above calculations $\norm{f\cdot\zeta_i-\zeta_i\cdot f}{E}\rightarrow 0$, for every $f\in C_c(\Gamma)$. Because $C_c(\Gamma)$ is dense in $F^*_{\lambda_p}(\Gamma)$, it follows that $\net{\zeta}{i}{I}$ is an almost $F^*_{\lambda_p}(\Gamma)$-central net. Thus, $E$ admits a non-zero central vector $\zeta$.

By Lemma \ref{lem:central_vectors_in_Fpstar_bimodule_from_iso_rep}, we have, for each pair $s,t\in \Gamma$, the equality
\[
\pi(s)\zeta(r)=\zeta(srs^{-1}).\numberthis\label{eqn:piactsonzetabyconjugationofinput}
\]
Take $t_0\in \Gamma$ such that $\zeta(t_0)\neq0$, and denote by $\mathrm{Cl}(t_0)=\set{tt_0t^{-1}}{t\in G}$ the conjugacy class of $t_0$. Since $\pi$ is an isometric representation of $\Gamma$ on $L^p(\Omega,\nu)$, we see from equation (\ref{eqn:piactsonzetabyconjugationofinput}) that $\norm{\zeta(tt_0t^{-1})}{L^p(\Omega,\nu)}=\norm{\zeta(t_0)}{L^p(\Omega,\nu)}$, for all $t\in \Gamma$. From this, we deduce that
\[
\norm{\zeta(t_0)}{L^p(\Omega,\nu)}^p\abs{\mathrm{Cl}(t_0)}=\sum_{r\in\mathrm{Cl}(t_0)}\norm{\zeta(r)}{L^p(\Omega,\nu)}^p\leq\sum_{r\in G}\norm{\zeta(r)}{L^p(\Omega,\nu)}^p=\norm{\zeta}{p}^p<\infty.
\]
Hence, as $\zeta(t_0)\neq0$, the set $\mathrm{Cl}(t_0)$ must be finite. 
Thus, the set $\pi(\Gamma)\zeta(t_0)$ is finite, and so, its span is a finite dimensional invariant subspace of $X$.

$(iii)\Rightarrow(iv)$:  Assume that $\Gamma$ does not have property $(\mathrm{T}_{L^p})$ and hence not property $(\rm T)$ (see Remark \ref{rk: T implies TLp}). Then, by work of Connes and Weiss, there exists a p.m.p  weakly mixing action on a probability space $(\Omega, \nu)$ admitting an asymptotically invariant sequence $(B_n)_n$ of measurable subsets of $\Omega$ with $\nu(B_n)=1/2$, for all $n$. Since the action is weakly mixing, the Koopman representation $\pi_0\colon \Gamma \rightarrow \mathrm{Iso}(L_0^p(\Omega, \nu))$ is weakly mixing due to Proposition \ref{prop: weak-mixing action and Koopman rep}. Hence, by Lemma \ref{lem: weak-mixing, no f.d. subrep}, $\pi_0$ does not have any finite dimensional subrepresentation. However, $\xi_n=2\chi _{B_n}-1$ provides an almost $\Gamma$-invariant sequence in $L_0^p(\Omega, \nu)$.
\end{proof}

\begin{rk}
One can prove the implication $(i)\Rightarrow(iii)$ in Theorem \ref{thm:TLp_for_discrete_gps} without passing by weak property $(\mathrm{T}_{SL^p})$ for $F^*_{\lambda_p}(G)$. Indeed, if $\Gamma$ has property $(\mathrm{T}_{SL^p})$ then any isometric representation on a closed subspace of an $L^p$-space with almost $\Gamma$-invariant vectors contains an invariant vector and hence a $1$-dimensional subrepresentation.
\end{rk}

\begin{rk}
In Theorem \ref{thm:TLp_for_discrete_gps}, the core of the proof of the implication $(ii)\Rightarrow(iii)$ is to show that the set $\mathrm{Cl}(t_0)$ is finite. Observe that, to reach this conclusion, it is shown that its Haar-measure is finite. The conclusion that $\mathrm{Cl}(t_0)$ is finite is therefore contingent on the discreteness of $\Gamma$. It does not seem feasible to us to extend this proof to a larger class of groups.
\end{rk}

\begin{rk}
    Theorem \ref{thm:TLp_for_discrete_gps} above remains true when exchanging $F^*_{\lambda_p}(\Gamma)$ with $F_{\lambda_p}(\Gamma)$.
\end{rk}

\begin{paragraph}{Acknowledgements}
The second named author was supported by Knut and Alice Wallenberg foundation through grant number 31001288.

We thank the organizers of the workshop on $C^*$-algebras and geometry of semigroups and groups held at the University of Oslo, where the initial ideas were discussed.
We are grateful to Nadia Larsen, Tim de Laat, N. Christopher Phillips, Mikael de la Salle and Alain Valette for comments on an earlier draft of this paper. The first named author thanks Nadia Larsen for many discussions and Mikael de la Salle for his insightful questions and suggestions and for discussions on subtleties in the different versions of property $(\mathrm{T}_{\mathcal{E}})$ for groups. Finally we are very grateful to the referee for the careful reading of our article and for very useful comments.
\end{paragraph}

% \newpage
\bibliography{bib}
\bibliographystyle{plain}

%%%%%%%%%%
	%% Address
%%%%%%%%%%
\vspace{2em}
\begin{minipage}[t]{0.45\linewidth}
  \small    Emilie Mai Elkiær\\
            Department of Mathematics\\
            University of Oslo\\
            0851 Oslo, Norway\\
		    {\footnotesize elkiaer@math.uio.no}
 \\   
	%	 Second author's address
 \\
 \small Sanaz Pooya \\
        Institute of Mathematics\\
        University of Potsdam\\
		14476 Potsdam, Germany\\
        {\footnotesize sanaz.pooya@uni-potsdam.de}
\end{minipage}

\end{document}